\documentclass{article}
\usepackage{amsthm,amssymb,amsmath}
\usepackage{color}
\usepackage{subcaption}
\usepackage{rotating, graphicx}

\usepackage{multirow}

 \newtheorem{theorem}{Theorem}

\newtheorem{exam}{Example}
\theoremstyle{definition}

\theoremstyle{remark}
\newtheorem{rem}{Remark}
\newtheorem{corollar}{Corollary}
\def \vol{{\rm Vol}}
\title{Induced optimal partition invariancy in linear optimization: constraints perturbation}
\author { Nayyer Mehanfar and Alireza Ghaffari-Hadigheh \\
Dept. of Applied Math. \\
 Azarbaijan Shahid Madani University, Tabriz, Iran\\
{\tt hadigheha@azaruniv.ac.ir}\\
{\tt mehanfar.n@azaruniv.ac.ir}
}
\date{}

\begin{document}
\maketitle
\begin{abstract}
    In this paper, we study uni-parametric linear optimization problems, in which simultaneously the right-hand-side and the left-hand-side of constraints are linearly perturbed with identical parameter.
 In addition to the concept of change point, the induced optimal partition is introduced, and its relation to the well-known optimal partition is investigated. Further, the concept of free variables in each invariancy intervals is described.
A modified generalized computational method is provided with the capability of identifying the intervals for the parameter value where induced optimal partitions are invariant. The behavior of the optimal value function is described in its domain.
Some concrete examples depict the results. We further implement the methodology on some test problems to observe its behavior on large scale problems.
	
\\ {\bf  MSC 2000}: 90c05,90c31
\end{abstract}

\section{Introduction}
   Existence of inaccuracy and variation in parameters of an optimization problem are indispensable, and investigation of their effects has attracted the primary concern of many researchers.   Deviation from a predetermined optimal solution would be an implication of such variation. Consequently, the current solution fails to be optimal or feasible, and additional burden of problem-solving for other values of parameters is imposed. Parametric programming has denoted its competence in this situation which economically identifies the exact mapping of the optimal solution in the space of parameters;  unnecessary many problems solving is avoided, and the optimal solution can be accordingly adjusted.

To be specific, consider the primal linear program
\begin{equation*}\nonumber\label{pp1}
    (P)~~~~~~\ \min \{ c^{T}x\ :\ Ax = b,\ x \geq 0 \},
\end{equation*}
and its dual as
\begin{equation*}\nonumber\label{pd1}
    (D)~~~~~~\ \max \{b^{T}y\ :\ A^{T}y + s = c,\ s \geq 0 \},
\end{equation*}
where $c \in \mathbb{R}^{n}$,  $b\in \mathbb{R}^{m}$ and $A \in \mathbb{R}^{m\times n}$  are fixed data; and  $x,s \in \mathbb{R}^{n}$ and $y \in \mathbb{R}^{m}$ are unknown vectors. Usually, vectors $b$ and $c$ are referred to as \emph{rim data}.
Assume that these problems are feasible, and denote their feasible solution sets by $\mathcal{P}$ and $\mathcal{D}$,
  respectively. A primal feasible solution $x$ and a dual feasible solution $(y,s)$ are optimal solutions if they satisfy the well-known complementary condition  $x^{T}s=0$.  Optimal solution sets of these problems are denoted by    $\mathcal{P}^{*}$  and  $\mathcal{D}^{*}$,  correspondingly.

  Consider the index set $\lbrace 1,2, \ldots , n \rbrace.$ Let $\mathcal{B}$ be a set of $m$ indices in this identifying a nonsingular submatrix $A_{\mathcal{B}}$ from the columns of~$A$, and $\mathcal{N}=\lbrace 1,2, \ldots , n \rbrace \setminus \mathcal{B}$. The vector $x^T=(x^T_{\mathcal{B}},x^T_{\mathcal{N}})$   is called a primal basic feasible  solution where $x_{\mathcal{B}}=A^{-1}_{\mathcal{B}}b\geq 0$, and $ x_{\mathcal{N}}=0$. A dual basic feasible  solution is identified as  $y=A^{-T}_{\mathcal{B}}c_{\mathcal{B}}$ and $s  = c - A^T A^{-T}_{\mathcal{B}}c_{\mathcal{B}}\geq 0$. Recall that a basic optimal solution specifies a partition of the index set, known as basic optimal partition.

A primal-dual optimal solution $(x^{*},y^{*},s^{*})$  is strictly complementary when  $x^{*}+s^{*}>0.$
By the Goldman-Tucker's theorem~\cite{GT56}, the existence of a strictly complementary solution is guaranteed if both Problems $(P)$ and $(D)$ are feasible. Interior point methods start from a solution and terminate at a strictly complementary optimal solution \cite{RT97}.
Having a primal-dual strictly complementary optimal solution, the index set $\lbrace 1,2, \ldots , n \rbrace$ is partitioned as
\begin{eqnarray*}
{B}&:=&\lbrace j\vert x^{*}_{j} > 0,\ {\rm{for~some}}\ x^{*}\in  \mathcal{P}^{*}    \rbrace,\\
{N}&:=&\lbrace j\vert s^{*}_{j} > 0, \ \ {\rm{for~some}}\ (y^{*},s^{*})\in  \mathcal{ِD}^{*}   \rbrace.
\end{eqnarray*}
This partition is denoted by $\pi =({B},{N})$,  and known as optimal \label{2pp} partition. It is identical with the basic optimal partition only if the primal and dual optimal solutions are nondegenerate.

We consider the following uni-parametric linear optimization problem
\begin{equation*}\label{mopt2m}
   P(\triangle  A,\triangle  b)~~~~~~\  \min \{c^{T}x :\ (A+\lambda  \triangle  A)x=b+\lambda \triangle  b,\ x \geq 0 \},
\end{equation*}
where $\lambda  \in \mathbb{R}$ is a parameter, $\triangle  b \in \mathbb{R}^{m}$ and $\triangle  A \in \mathbb{R}^{m\times n}$ are perturbing directions with no specific restriction on  $\triangle  A$ and $\triangle  b$.

In uni-parametric programming,  it is assumed that the data are perturbed along a direction according to the parameter. The aim is to identify the region for the parameter value, where specific properties of the current optimal solution do not alter.  For example, if the current optimal solution is basic, one might be interested in identifying the region at which the known optimal basis is invariant for every parameter values in this region. On the other hand, when the solution is strictly complementary, the aim could be the specifying of a region at which the known optimal partition is invariant for every parameter value in this region.
We characterize conditions that guarantee the convexity of these regions.
 Recall that degeneracy of a primal or dual basic solution leads to having multiple dual or primal optimal solutions. Therefore, the basic optimal partition may not be unique in general. In spite, the convexity of optimal solution sets implies the uniqueness of the optimal partition.

We recall that the Problem $P(\triangle  A,\triangle  b)$ has been studied when the basic optimal partition is known. The case when either $\triangle  A$ or $\triangle  b$ is zero, also studied.
In this paper, we investigate the behavior of Problem $P(\triangle  A,\triangle  b)$ when the optimal partition is known, and identify the regions where the given optimal partition is invariant. To do this end, we have to define the notion of induced optimal partition. Despite the optimal partition, which defines transition points, induced optimal partition suggests other points; here, we refer to as change points. These two types of points lead to determine the regions as intervals.

  A computational algorithm is provided for finding all potential such intervals. Furthermore, the representation of the optimal value function is presented, and the results are clarified using some concrete examples. The methodology is also implemented on some test problems to investigate the efficiency of the algorithm and to observe computational challenges.

 The rest of paper is organized as follows. Section \ref{Literature} is devoted to reviewing the so-far-existing results of linear parametric programming in a nutshell. In Section 3, in addition to expression the admissibility of a direction $(\triangle  A, \triangle b)$ of $P_{\lambda}(\triangle A,\triangle b),$ we review some other necessary concepts as optimal partition invariancy on Problem $P_{\lambda}(\triangle A),$ pseudo-inverse, and realization theory. In Section~\ref{sec3},  the concept of optimal partition invariancy on principal problem is generalized  to the  Problem $P_{\lambda}(\triangle A)$. The notion of change point is defined and distinguished from the transition point. Its economical interpretation is highlighted via a simple example.  In Section 5, a generalized explicit formula is presented that identifies a closed form of the optimal value function on each invariancy interval. The process of finding all transition points,  change points, and invariancy intervals are provided in Section \ref{sec6}. Some concrete examples illustrate the results in Section  \ref{sec7}, and the methodology is also implemented on some test problems. The final section contains some concluding remarks.

\section{  Literature Review}\label{Literature}

Let us first review some findings in parametric linear programming. In a parametric linear optimization when right-hand-side is perturbed, it is proved that the invariancy regions are open intervals if they are not a singleton. Moreover, the optimal value function is a continuous piecewise linear over these intervals~\cite{RT97}, and therefore, these intervals may be referred to as linearity intervals. Singleton regions are referred to as break points since optimal value function does not have a continuous derivative at these points.

Perturbation at the left-hand-side alongside a direction $\triangle A$ with arbitrary rank has been studied to find basic invariancy intervals\cite{FR85}.  The author studied properties of the optimal value function and the associated optimal basic feasible solution
in a neighborhood about a fixed parameter value. Computable formulas for the Taylor series' (and hence all derivatives) of the optimal value function were provided, as well as for the primal and dual optimal basic solutions as functions of the parameter. It also identifies the parameter values where the optimal value function is differentiable in the case of degeneracy.

In another study, the author considered the same problem and could present a closed form of the optimal value function as a fractional function of the parameter in the optimal basic invariancy interval \cite{zuidwijk2005linear}.
  It was denoted that the optimal value function is piecewise fractional in terms of the parameter.   In another study, a solution algorithm has presented for a linear program with inequality constraints and a single parameter at their left-hand-side \cite{KK14}. This algorithm includes inversion techniques of perturbed matrices, which accompanies with some computational complexities.  The case when the problem is in the canonical form, a primal-dual strictly complementary optimal solution is in hand, and perturbing direction is of rank one, has been studied in \cite{HG99}.

Recently an algorithm has been introduced for the exact solution of multi-parametric linear programming problems with inequality constraints~\cite{CL17}. The perturbation occurs simultaneously on the objective function’s coefficients, the right-hand-side, and the left-hand-side of the constraints. This algorithm is based on the principles of symbolic manipulation and semi-algebraic geometry, and critical regions are identified by semi-algebraic geometry. From the critical region, the authors meant the region where active constraint sets are invariant.  , It is shown that these regions are neither necessarily convex nor connected.   The entire parametric space is explored implicitly within the algorithm, while there is no necessity for determining the inverse of parametric matrices. It was shown that the objective function is fractional on critical regions.  
  Though their considered problem can be reduced to $P(\triangle A,\triangle  b ),$ their assumption on the invariancy of the active set of constraints necessitate having an optimal basic solution. Moreover, their approach is highly dependent on mathematical software, which may increase the complexity of computation. As the author acknowledged,  their methodology is not efficient for large-scale problems.

In another study, the case when $\triangle A$ is of arbitrary rank, and the problem is in standard form with known optimal partition, has been investigated \cite{GM15}. Their method was an adaptation of the proposed approach in~\cite{zuidwijk2005linear}.
The authors presented a computational procedure for finding all intervals, and the representation of the optimal value function on these intervals. They further investigated some properties of the optimal value function.
 Let us call the considered problem in~\cite{GM15} as the principal parametric linear program  which, in our notation,  is
\begin{equation}\nonumber
   P_{\lambda }(\triangle A) ~~~\ \min \{c^{T}x\ :\ (A+\lambda  \triangle  A)x=b,\ x \geq 0 \}.
\end{equation}
Its dual is
\begin{equation}\nonumber
   D_{\lambda }(\triangle A) ~~~\ \max \{b^{T}y\ :\ (A+\lambda  \triangle  A)^{T}y  + s =c,\ s \geq 0 \}.
\end{equation}
 Let $\mathcal{P}^{*}(\lambda )$  and  $\mathcal{D}^{*}(\lambda )$  denote the optimal solution sets of  $P_{\lambda }(\triangle A)$ and $D_{\lambda }(\triangle A)$, respectively.

The main results deserve to mention are as follows. Despite the case when only rim data is perturbed,  that the domain of the optimal value function is closed~\cite{ART7, RT97}, this domain might be open when $\triangle A$ is perturbed. Further, there is no clear relation between optimal partitions at a breakpoint and at its potential neighboring intervals despite the perturbation at rim data~\cite{AH10}. Moreover, the invariancy region might not be convex unless with some conditions\cite{HG99}.
  The last, unlike rim variation, the set of admissible changes (See page \pageref{paadd}) corresponding to Problem $P_{\lambda }(\triangle A)$ is not a convex set~\cite{HG99}.
\section{Preliminaries}
In this section, some necessary preliminary concepts and assumptions for the convexity of the invariancy region are posed, and optimal partition invariancy is reviewed in our literature. It follows by the concept of Moore-Penrose inverse and finalizes with some facts from realization theory.
\subsection{Convexity of the invariancy region}
 In the following subsection, we consider some assumptions that guarantee the convexity of the invariancy region for the problem $P(\triangle  A,\triangle  b).$
Observe that the  dual of this problem is \label{paadd}
\begin{equation*}\label{moptd3m}
  D(\triangle  A,\triangle b) ~~~~ \  \max \{(b+\lambda  \triangle  b )^{T}y :\ (A+\lambda  \triangle  A)^{T}y +s = c, s \geq 0 \}.
\end{equation*}
Here  $\lambda  $ runs throughout a nonempty subset $\Lambda \subseteq \mathbb{R}$, where  Problem $P(\triangle  A,\triangle  b)$ has an optimal solution for every parameter value $\lambda \in \Lambda $. This set is nonempty since it is assumed that  Problems $P(\triangle  A,\triangle  b)$ and $D(\triangle  A,\triangle  b)$ are feasible at $\lambda =0 $. In general,  $\Lambda $ is not a convex set, and  to guarantee its convexity, we have to make some assumptions.  First, let  $\triangle   =(\triangle  A, \triangle  b) $ denote a perturbing direction referred to as \emph{change direction}. For a $\lambda  \in \mathbb{R}$,
$\lambda \triangle$ is said an \emph{admissible change} if  Problem $P(\triangle  A,\triangle  b)$  has an optimal solution, or equivalently
\begin{equation}\nonumber
  \exists (x,y), x\geq 0 :\ (A+\lambda  \triangle  A)x=b+\lambda  \triangle  b ,\ (A+\lambda  \triangle  A)^{T}y\leq c.
\end{equation}
There are some instances at which,  $\lambda   \triangle$  is not an admissible change for all $\lambda  \in  (0,\lambda  ^{*})$, just because $\lambda ^*  \triangle$ is an admissible change.  A change direction $\triangle$ is an \emph{admissible direction} if there exists $\lambda  ^{*}>0$, such that $\lambda  \triangle $ is an admissible change for all $\lambda  \in [0,\lambda  ^{*}).$
 Let us denote the set of all admissible directions by $\mathcal{A}$.
For    an admissible direction $\triangle$, let $\Lambda (\triangle ) := \{ \lambda  : \ \lambda  \triangle   \in \mathcal{A} \}$ and $\lambda  ^{*} (\triangle ) := \sup \{\lambda ^{* } : \ \forall \lambda  \in  [0,\lambda  ^{*}), \lambda \in \Lambda (\triangle ) \}.$
Analogous to  \cite{HG99}, it can be  proved that if $\mathcal{A} = \bigcup _{k=1}^{K} \{ \mathcal{P}_{k}\}$, and each $\mathcal{P}_{k}$ is a polyhedron containing the origin, then  $\Lambda (\triangle  )$ is simply an interval. In this study, we assume that these conditions are fulfilled.

\subsection{Optimal partition invariancy } \label{optin}
Let us quote the notion of optimal partition from~\cite{GM15}, and clarify its relationship with the optimal partition defined in Page~\pageref{2pp}.
Let
\begin{equation}\label{taus}
  \tau  :\{1,\ldots,l\} \rightarrow \{1,\ldots,n\},  { {\tau }'}:\{1,\ldots,n-l\} \rightarrow \{1,\ldots,n\},
\end{equation}
 be injective and strictly increasing functions, where $ 1\leq  l  \leq \max  \{m,n\}$, and
${\rm Range}(\tau  )\cup {\rm Range}({{\tau  }'})=\{1,\ldots,n\}$. Further, let 
\begin{equation}\label{eb}\nonumber
{\rm E}(\tau  ) = (e_{{\tau } (1)} \ \ldots \ e_{{\tau } (l)}): {\mathbb R} ^l \rightarrow {\mathbb R}^n,
\end{equation}
be a map where $A_{\tau  }= A {\rm E}(\tau  ):{\mathbb R} ^l \rightarrow {\mathbb R}^m$ and $c_{\tau }={\rm E}(\tau  )^Tc\in {\mathbb R}^l $. Let $\lambda _{0} \in \mathbb R$ such that $(x^{*}(\lambda _{0}), y^{*}(\lambda _{0}),$  $s^{*}(\lambda _{0}))$ is a strictly complementary optimal solution of Problems $P_{\lambda _{0}}(\triangle A)$ and $D_{\lambda _{0}}(\triangle A) $ with the optimal partition $\pi (\lambda  _{0}) =({B}(\lambda  _{0}),{N}(\lambda  _{0}))$, where
\begin{eqnarray*}
{B}(\lambda  _{0})&:=&\lbrace j\vert x^{*}_{j}(\lambda  _{0}) > 0,\ {\rm{for~some}}\ x^{*}(\lambda _{0})\in  \mathcal{P}^{*}({\lambda  _{0}})    \rbrace,\\
{N}(\lambda  _{0})&:=&\lbrace j\vert s^{*}_{j} (\lambda  _{0})> 0, \ \ {\rm{for~some}}\ (y^{*}(\lambda  _{0}),s^{*}(\lambda  _{0}))\in  \mathcal{D}^{*} ({\lambda  _{0}})  \rbrace.
\end{eqnarray*}
Moreover, let $\Lambda_{\pi (\lambda  _{0})} \subseteq \mathbb{R}$ be a set of parameters such that for each $\lambda   \in \Lambda_{\pi (\lambda  _{0})} $, the principal Problem $P_{\lambda }(\triangle A)$ has an optimal solution $(x^{*}(\lambda  ), y^{*}(\lambda  ), s^{*}(\lambda  ))$ with the  optimal partition  $\pi (\lambda  _{0}) =({B}(\lambda  _{0}),{N}(\lambda  _{0}))$. Clearly, any strictly complementary optimal solution corresponding to an arbitrary $t \in \Lambda _{\pi (\lambda _{0})}$  implies $ {B} (\lambda ) = \lbrace \tau (1), \ldots , \tau (l) \rbrace$ and $ {N} (\lambda )= \lbrace {\tau }' (1), \ldots , {\tau }' (n- l) \rbrace$. Furthermore, when  the following three properties hold for $\lambda \in \Lambda_{\pi (\lambda _0)}$ \label{proper13}
\begin{description}
\item[\textbf{Property 1.}]  $A_{\tau }(\lambda )$ has a pseudo-inverse,
 \item[\textbf{Property 2.}]   $x_{\tau}(\lambda )=A_{\tau}(\lambda )^{\dagger}b> 0$,
  \item[\textbf{Property 3.}]    $s_{{\tau}'}(\lambda ) >0$ (or equivalently $c_{{\tau}'}^T-c_{\tau }^TA_{\tau}(\lambda)^{\dagger}A_{{\tau}'}(\lambda )> 0$), \label{c3n}
\end{description}
   then $(x(\lambda ), y(\lambda ), s(\lambda ))$ is a strictly complementary optimal solution of  Problems $P_{\lambda }(\triangle A)$ and $D_{\lambda }(\triangle A) $ with optimal partition $\pi (\lambda _0)$.
   Recall that the aim of optimal partition invariancy is to find the region $\Lambda_{\pi (\lambda _0)}$, where for every  $\lambda   \in \Lambda_{\pi (\lambda _0)}$, optimal partition of the associated problem is $\pi (\lambda _0) $. This is equivalent to establishment of  Properties 1-3 for such  a parameter $\lambda $.

\subsection{Moore-Penrose  inverse}\label{sec2}
There exists a unique matrix for a real matrix $A \in {\mathbb R}^{m \times n} $,  is known as Moore-Penrose  inverse (or simply pseudo-inverse) and denoted by $A^{\dagger} \in {\mathbb R}^{n \times m}$, that satisfies the following equations.
\begin{eqnarray}\label{0d}
A^{\dagger}AA^{\dagger}&=&A^{\dagger},\\ \label{1d}
AA^{\dagger}A&=&A,\\ \label{2d}
(A^{\dagger}A)^{*}&=&A^{\dagger}A,\\\label{3d}
(AA^{\dagger})^{*}&=&AA^{\dagger},
\end{eqnarray}
where $A^{*}$ denotes the conjugate transpose of  matrix $A$.
 In general, $AA^{\dagger}$ is not necessarily an identity matrix, while it maps all column vectors of $A$ to themselves, and $(A^{\dagger})^{\dagger} = A$. A matrix satisfying (\ref{0d}) and (\ref{1d}) is known as a generalized inverse. It  always exists but is not usually unique unless  Conditions (\ref{2d}) and (\ref{3d}) hold, too. Observe that if a matrix is nonsingular, the pseudo-inverse and the inverse coincide. For more details, one can refer to~\cite{ben2003generalized}.

Let $\mathcal{H}(A)$ \label{1nm} be an index set of nonsingular $r \times r$ submatrices $A_{IJ}$ of $A\in {\mathbb R}^{m \times n}$ with ${\rm{Rank}}(A)=r$. It was { proved~\cite{ben1992volume}} that the pseudo-inverse $A^\dag$  is a convex combination of ordinary inverses $\{A_{IJ}^{-1}\ :\ (I,J)\in \mathcal{H}(A)\}$, as
\begin{equation}\label{1i}\nonumber
    A^\dag=\displaystyle \sum_{(I,J)\in {\mathcal {H}}(A)} t_{IJ}\widehat{A_{IJ}^{-1}},
\end{equation}
where $\widehat{A_{IJ}^{-1}}$ denotes that $A_{IJ}^{-1}$  padded with the right number of zeros in the right places, and
\begin{equation}\label{4i}\nonumber
    t_{IJ}=\frac{{\det }^2A_{IJ}}{{\displaystyle \sum_{(\bar{I},\bar{J})\in {\mathcal H}(A)}}{\det }^2 A_{\bar{I}\bar{J}}},\  (I,J)\in \mathcal{H}(A).
\end{equation}
 The volume of $A$, denoted by $\vol ~A$, is defined as  0 if $r=0$ and
\begin{equation}\label{5i}\nonumber
    \vol~A := \sqrt{{\displaystyle \sum_{(I,J)\in {\mathcal H}(A)}{\det }^2 A_{IJ}}},
\end{equation}
otherwise~\cite{ben1992volume}.
Let    $b_I$ be a subvector of $b\in {\mathbb R}^m$ with indices in $I$. The Euclidean norm least squares solution of the linear system $Ax=b$ is a convex combination of basic solutions $A_{IJ}^{-1} b_I$, i.e.,
\begin{equation}\label{3i}
    x=A^\dag b= \displaystyle \sum_{(I,J)\in {\mathcal {H}}(A)} t_{IJ}\widehat{A_{IJ}^{-1} b_I}.
\end{equation}
  The representation (\ref{3i}) follows easily for $A$ of full column rank~\cite{BT90}.
When $A$ is a matrix of full row rank, a solution $x$ is  in the row space of $A,$ i.e.,   $x=A^{T}(AA^{T})^{-1}b$.
\subsection{Realization theory}
Here, the realization theory for scalar rational functions is reviewed. More details for   matrix-valued and operator-valued functions can be found  in~\cite{BGK79}. A fundamental observation in realization is that, when $b,c\in {\mathbb R}^l$ and $C\in {\mathbb R}^{l\times l},$ then   $f(\lambda)=1+ \lambda c^T(I_l+\lambda C)^{-1}b$
is a rational function and can be described completely in terms of eigenvalues of two matrices $C$ and $C^{\times}=C+bc^T$ as
\begin{eqnarray*}
f(\lambda) &=& \det f(\lambda ) = \det (1+\lambda c^T(I_l +\lambda C)^{-1}b)=\det (I_l +bc^T(I_l+\lambda C)^{-1})\\
&=&\frac {\det (I_l+\lambda (C+bc^T))}{\det (I_l +\lambda C)} = \frac {\det (I_l +\lambda C^{\times})}{\det (I_l +\lambda C)},
\end{eqnarray*}
{ where   $I_{l}$ is an $l {\times} l$ identity matrix. }
 More explicitly, let $\alpha_1,\ldots ,\alpha_l$ be eigenvalues of $C$ and $\alpha ^{\times}_1,\ldots ,\alpha ^{\times }_l$ be eigenvalues of $C^{\times}$, counted according to their multiplicities. Then
\begin{equation}\label{1}
    f(\lambda)=\displaystyle \prod_{j=1}^{l} \frac{1+\lambda \alpha ^{\times }_j}{1+\lambda \alpha_j}.
\end{equation}
The number $l$, i.e., the  factors in numerator and denominator, on the right-hand-side of (\ref{1}) is minimal when $C$ and $C^{\times }$ do not have  common eigenvalues.

\section{New concepts and main results}\label{sec3}
In this section, we present our approach for finding the invariancy interval of Problem $P( \triangle  A,\triangle  b ).$ The first step is to convert this problem to the one with only perturbation in the left-hand-side of the constraints.

Recall that two problems are  equivalent if they contribute similar features, and  a solution of one  can be immediately identified by the other's, while they may have different numbers of variables and constraints~\cite{BV04}.
Let us rewrite $P( \triangle  A,\triangle  b )$ as
\begin{equation}\label{mop21c12} \nonumber
  \min \{c^{T}x\ :\ Ax+ \lambda (\triangle  A x -\triangle  b)= b,\  x\geq 0 \}.
\end{equation}
Considering  $z =\triangle  Ax - \triangle  b,$ as a constraint,  leads to a problem in which only the coefficient matrix is  perturbed with  $\lambda $
 as
\begin{equation}\nonumber
   P_{\lambda }(\mathbf{\triangle  A}) \ \ \min \{\mathbf{c}^{T}\mathbf{x}\ :\ (\mathbf{A} + \lambda  \mathbf{\triangle  A}) \mathbf{x} = \mathbf{b}, \ \mathbf{x}= (x,z)^{T}
, \ x \geq 0 \},
\end{equation}
with
\begin{equation}\nonumber
\mathbf{c}=\begin{pmatrix}
c \\
0
\end{pmatrix},
\mathbf{A}=\begin{pmatrix}
A & 0  \\
\triangle  A & -I_{m}
\end{pmatrix},
  \mathbf{\triangle  A}= \begin{pmatrix}
0 & I_{m}\\
0 & 0
\end{pmatrix},
  \mathbf{b}=\begin{pmatrix}
b \\
\triangle  b
\end{pmatrix},
\end{equation}
where   $I_{m}$ is an $m \times m$ identity matrix, zeros are matrices of appropriate sizes,  $\mathbf{x}_{j}:=x_{j}, j=1, \ldots ,n $, and $\mathbf{x}_{n+i}:=z_{i}, i=1, \ldots ,m$.
The optimal value function of this problem is denoted by $Z(\lambda)$.
The dual of $P_{\lambda }(\mathbf{\triangle  A})$ is
\begin{equation}\nonumber
   D_{\lambda }(\mathbf{\triangle  A}) \ \ \max \{b^{T}y + \triangle  b^{T} w  :\ A^{T} y + \triangle  A ^{T} w +s= c, \ \lambda  y - w=0,\ s\geq 0 \}.
\end{equation}
 Suppose $\mathcal{P}_{\lambda }^{*}(\mathbf{\triangle  A})$ and  $\mathcal{D}_{\lambda }^{*}(\mathbf{\triangle  A})$, respectively denote  optimal solution sets of  $P_{\lambda }(\mathbf{\triangle  A})$ and $D_{\lambda }(\mathbf{\triangle  A}).$ Further, let $\sigma (v) = \lbrace i \vert {v}_{i} \neq 0 \rbrace $ as the support set of an arbitrary vector $v$.  Let  ${\mathbf{s}^{*{T}}(\lambda _{0})}=(s^{{*}{T}}(\lambda _{0}),0),$ where zero is a row vector with dimension $m,$ and $ {\mathbf{y}^{*T}(\lambda  _{0})}=(y^{*T}(\lambda  _{0}),w^{*T}(\lambda  _{0}))$. We define the partition $\bar{\pi}(\lambda  _{0}) =({\bar{B}}(\lambda  _{0}),{\bar{N}}(\lambda  _{0}))$ as   $\bar{B}(\lambda  _{0}) :=  {B}(\lambda _{0}) \cup {B^{-}}(\lambda  _{0}) \cup {B^{+}}(\lambda  _{0})$    with
   \begin{eqnarray*}
     B(\lambda  _{0}) &:=& \lbrace j \vert {\mathbf{x }_{j}^{*}}(\lambda  _{0}) >0,1\leq   j \leq n,  \ {\rm{for~some}}\ \mathbf{x }^{*}(\lambda  _{0})\in  \mathcal{P}_{\lambda _{0} }^{*}(\mathbf{\triangle  A}) \rbrace,\\
     {B^{-}}(\lambda  _{0}) &:=& \lbrace j \vert {\mathbf{x }_{j}^{*}}(\lambda  _{0}) <0,n+1\leq   j \leq n+m, \ {\rm{for~some}}\ \mathbf{x }^{*}(\lambda  _{0})\in  \mathcal{P}_{\lambda  _{0}}^{*}(\mathbf{\triangle  A}) \rbrace,\\
     {B^{+}}(\lambda  _{0})& :=&\lbrace j \vert {\mathbf{x }_{j}^{*}}(\lambda  _{0}) >0,n+1\leq   j \leq n+m, \ {\rm{for~some}}\ \mathbf{x }^{*}(\lambda  _{0})\in  \mathcal{P}_{\lambda  _{0}}^{*}(\mathbf{\triangle  A}) \rbrace,
        \end{eqnarray*}
   and ${\bar{N}}(\lambda  _{0}) :={N}(\lambda  _{0}) \cup {N^{0}}(\lambda  _{0})$ with
    \begin{eqnarray*}
      {N}(\lambda  _{0}) &:=&\lbrace j \vert {\mathbf{s }_{j}^{*}}(\lambda  _{0}) >0, 1\leq   j \leq n,  \ {\rm{for~some}}\ (\mathbf{y }^{*}(\lambda  _{0}),\mathbf{s }^{*}(\lambda  _{0}))\in  \mathcal{D}_{\lambda _{0} }^{*}(\mathbf{\triangle  A}) \rbrace,\\
      {N^{0}}(\lambda _{0}) &:=&\lbrace j \vert {\mathbf{x }_{j}^{*}}(\lambda  _{0}) =0, n+1\leq   j \leq n+m,  \ {\rm{for~{all}}}\ \mathbf{x }^{*}(\lambda _{0})\in  \mathcal{P}_{\lambda  _{0}}^{*}(\mathbf{\triangle  A})\rbrace .
    \end{eqnarray*}
   It can be easily understood that $B(\lambda _0)$     in this partition is the same as in the optimal partition extracted from a strictly complementary optimal solution of Problems $P( \triangle  A,\triangle  b )$ and  $D( \triangle  A,\triangle  b)$ at $\lambda  _{0}$. Moreover, considering the equivalence of two Problems $P( \triangle  A, \triangle  b )$ and $ P_{\lambda  }(\mathbf{\triangle  A}), $ an optimal solution $(\mathbf{x}^{*}(\lambda  ),\mathbf{y}^{*}(\lambda ) ,\mathbf{s}^{*}(\lambda  ))$ can be induced by an optimal solution  $x^{*}(\lambda  )$ of Problem $P( \triangle  A, \triangle  b )$ and vice versa. Therefore, we refer to such an optimal solution $\mathbf{x}^{*}(\lambda )$ for $ P_{\lambda }(\mathbf{\triangle  A}) $ as induced optimal solution, and   the  corresponding partition as induced optimal partition.
    Note that, an induced optimal solution $(\mathbf{x}^{*}(\lambda  ),\mathbf{y}^{*}(\lambda  ),\mathbf{s}^{*}(\lambda ))$ is an induced strictly complementary solution when for all $j\in \{1,\ldots ,n \},\ \mathbf{x}_{j}^{*}(\lambda  ) +\mathbf{s}_{j}^{*}(\lambda  )>0,$ and for $n+1 \leq j \leq m+n,\ \mathbf{x}_{j}^{*}(\lambda )$ and $\mathbf{s}_{j}^{*}(\lambda  )$ are not simultaneously zero.

In the sequel, we assume that $\pi (\lambda _{0}) =({B}(\lambda _{0}),{N}(\lambda _{0}))$ is the known optimal partition of Problems $P( \triangle  A, \triangle  b )$ and $D(\triangle  A, \triangle  b)$ at $\lambda _{0}$,
and $\bar{\pi}(\lambda _{0}) =({\bar{B}}(\lambda _{0}),{\bar{N}}(\lambda _{0}))$ is the induced optimal partition
 of Problems $P_{\lambda _{0} }(\mathbf{\triangle  A})$ and  $D_{\lambda _{0} }(\mathbf{\triangle  A}).$ Let $l=\vert {B}(\lambda _{0}) \vert $ and $ \bar{l}=\vert {\bar{B}}(\lambda _{0}) \vert .$ Obviously, $l \leq \ n$ and $l \leq \bar{l}\leq m+n$.
Let us adapt the notations and concepts of the principal Problem $P_{\lambda }(\triangle A)$ to the Problem $P_{\lambda }(\mathbf{\triangle  A})$.
 Let
 $$\begin{array}{l}
   \bar{\tau} :\{1,\ldots,l, \ldots, \bar{l}\} \rightarrow \{1,\ldots,n,\ldots , n+m\},\\
   {\bar {\tau}}^{'}:\{1,\ldots ,(n+m)- \bar{l}\} \rightarrow \{1,\ldots,n,\ldots,n+m\},
 \end{array}$$
 be injective and strictly increasing functions
with ${\rm Range}({\bar {\tau}})\cup {\rm Range}({\bar {\tau}}^{'})=\{ 1,\ldots , n+m\}$.
Analogously, define
\begin{equation}\nonumber
  {\rm E}(\bar{\tau}) = (e_{{\bar{\tau}} (1)} \ \ldots \ e_{{\bar{\tau}} (l)}\ \ldots \ e_{{\bar{\tau}} (\bar{l})}): {\mathbb R} ^{\bar{l}} \rightarrow {\mathbb R}^{(n+m)}.
\end{equation}
 It can be easily observed that $ {\bar{B}}(\lambda _{0})= \lbrace \bar{\tau}(1), \ldots ,\bar{\tau}(l), \ldots , \bar{\tau}(\bar{l}) \rbrace $ and $ {\bar{N}}(\lambda _{0})= \lbrace {\bar{\tau}}^{'}(1), \ldots , {\bar{\tau}}^{'}(n+m-\bar{l}) \rbrace.$
Moreover, when $ \lbrace i |n+1\leq i \leq  m+n,   {v}_{i} \neq 0   \rbrace$ is empty then $l =\bar{l},$ i.e., $ {\bar{B}}(\lambda _{0})={B}(\lambda _{0})= \lbrace \tau (1), \ldots ,\tau (l) \rbrace .$
 In this case,
\begin{equation}\nonumber
 {\mathbf{x}}^{*}_{j}(\lambda _0) \left\{ \begin{array}{ll}
> 0, & \mbox{if} \ j \in \lbrace \bar{\tau}(1), \ldots , \bar{\tau} (l)\rbrace ,\\ \\
= 0, & \mbox{if} \ j \in \lbrace {\bar{\tau}}^{'}(1), \ldots , {\bar{\tau}}^{'} (n+m-l)\rbrace.
\end{array}
\right.
\end{equation}
On the other hand, for $l<\bar{l}$,
\begin{equation}\nonumber
{\mathbf{x}}^{*}_{j}(\lambda _0) \left\{ \begin{array}{ll}
> 0, & \mbox{if} \ j \in \lbrace \bar{\tau}(1), \ldots , \bar{\tau} (l)\rbrace ,\\ \\
\neq 0, & \mbox{if} \ j \in \lbrace \bar{\tau}(l+1), \ldots , \bar{\tau} (\bar{l})\rbrace ,\\ \\
= 0, & \mbox{if} \ j \in \lbrace {\bar{\tau}}^{'}(1), \ldots , {\bar{\tau}}^{'} (n+m-\bar{l})\rbrace.
\end{array}
\right.
\end{equation}
Further, for the induced optimal solution $({\mathbf{y}}^{*}(\lambda _{0}),{\mathbf{s}}^{*}(\lambda _{0}))$ of Problem $ D_{\lambda _0}( \mathbf{\triangle  A})$, it holds
\begin{equation}\nonumber
 {\mathbf{s}}^{*}_{j}(\lambda _0) \left\{
 \begin{array}{ll}
 > 0, & \mbox{if}~ \ j \in {\bar{N}}(\lambda _{0})\\
= 0, &  \mbox{otherwise.}
\end{array}
\right.
\end{equation}
The following theorem states { necessary and sufficient conditions} for an optimal solution of $P_{\lambda _0}( \mathbf{\triangle  A})$ being an induced optimal solution.
\begin{theorem} \label{thg}
  Let $\lambda \in \Lambda ,$ and $\bar{\pi}(\lambda ) =({\bar{B}}(\lambda ),{\bar{N}}(\lambda )).$
   Then, $({\mathbf{x}}^{*}(\lambda ), {\mathbf{y}}^{*}(\lambda ), {\mathbf{s}}^{*}(\lambda ))$ is an induced strictly complementary optimal solution of $ P_{\lambda }( \mathbf{\triangle  A})$ and $ D_{\lambda }( \mathbf{\triangle  A})$
     if and only if
\begin{description}
\item{{\rm Cond. 1.}} $\mathbf{A}_{\bar{\tau }}(\lambda )$ has pseudo-inverse,
\item{{\rm Cond. 2.}}
For $ 1\leq q \leq \bar{l}$, $({\mathbf{x}}_{\bar{\tau}}^{*}(\lambda ))_{q} =e^{T}_{q}{\mathbf{A}}_{\bar{\tau}}^{\dagger}(\lambda ) \mathbf{b} $ is positive when $q \in {B} \cup {B^{+}}$, negative when $q \in {B^{-}}$, and zero otherwise,
\item{{\rm Cond. 3.}}
 ${\mathbf{s}}_{{\bar{\tau}}^{'}}^{*}(\lambda )= s_{{\tau }'}^{*}(\lambda ) = c_{ {\tau}'}^T-c_{\tau }^T A^{\dagger}_{\tau}(\lambda )A_{{\tau}'}(\lambda )> 0,$ and zero otherwise. \label{nbn}
\end{description}
\end{theorem}
\begin{proof}
Observe that satisfaction of Cond.~1 is  necessary to others, since the pseudo-inverse of $\mathbf{A}_{\bar{\tau}}(\lambda )$ appears in them. Moreover,    Cond.s 2 and 3, respectively, are the strictly feasibility conditions of the solution for Problems $ P_{\lambda }( \mathbf{\triangle A}) $ and $ D_{\lambda }( \mathbf{\triangle  A}).$ Since for $\lambda  $, the dual feasibility condition of  $ D_{\lambda } (\mathbf{\triangle  A})$   only associates  to  those variables with indices in ${\tau}'(i)$ ,  $\ i \in \lbrace 1, \ldots , n-l \rbrace $, it  suffices to consider the indices of ${\tau }'$ instead of ${\bar{\tau }}^{'}$. Recall that ${\tau }'(i)$  corresponds to the positive variables $s_{i}$ in Problem $D( \triangle  A , \triangle b)$ where $\ i \in \lbrace 1, \ldots , n-l \rbrace $.
Respecting the concept of induced optimal solution and  strictly feasibility of primal and dual problems, validity of the statement   is immediate.
\end{proof}

Now,  having the induced optimal partition $\bar{\pi}(\lambda _0) =({\bar{B}}(\lambda _0),{\bar{N}}(\lambda _0))$, for a given $\lambda \in \Lambda _0,$ the aim of induced optimal partition invariancy is to find the region $\Lambda_{\bar{\pi }}(\lambda _0)\subseteq \Lambda$, where for all  $\lambda  \in \Lambda_{\bar{\pi }}(\lambda _0)$, the induced optimal partition of the associated problem is identical with $\bar{\pi }(\lambda _0)$. Recall that this invariancy region contains all such $t $'s at which Cond.s 1-3  in Theorem \ref{thg} hold.

Let us define the notion of \emph{change point} to distinguish it from the transition point.
Due to the assumption that ${\triangle  \mathbf{A}_{\bar{\tau}}}$ is an admissible direction (See page~\pageref{paadd}), the induced optimal partition invariancy region, is an interval.
At the endpoints of this interval,  induced optimal partitions are changed provided that the Problem $ P_{\lambda }( \mathbf{\triangle A}) $ is feasible and has optimal solutions at these points. Otherwise, there is no induced optimal partition at these points.
Variation of induced partition means that some indices interchange between  $\bar{B}$ and $\bar{N}$ when the parameter is replaced by one of the endpoints of the interval. More clearly, this transition may happen between  $B$ and $N$, or between $B^+$, $B^-$ and $N^0.$ In the first case, the point is referred to as a \emph{transition point}, and in the second case, it is called a \emph{change point} \label{chpoint}.
Note that when a parameter value is a change point only, indices of free variable interchange between their index sets only. Since they are absent in the objective function $Z(\lambda )$,  this variation does not affect its representation.
Therefore, at a transition point, the representation of the optimal value function changes, and it fails to have the first derivative. The representation of the optimal value function on the neighborhoods of the change point does not change when it is not a transition point, simultaneously.

Let us illustrate the importance of a change point in practice via a simple example.
Consider manufacturing of $n$ products, using $m$ sources, and $b_{i}$ denotes the amount of available value of source $i.$ A unit of product $j$ needs $a_{ij}$ amount of the source ${i}.$ For instance, let $b_{1}$ is the available machine time in this production plan, and $a_{1j}$ is the time necessary to produce one unit of item $j$. Therefore, the corresponding constraint could be as
$$a_{11} x_1+\cdots + a_{1n} x_n \leq b_{1}$$
where $x_{j}$ is the production level of $j$. Without loss of generality, one may assume that variation in production time affects its quality. Thus, positive $\triangle a_{1j},$ as the increase in the production time of $j$ increases the quality of this product and vice versa. This variation may necessitate a change on $b_{1}$, but this is not the only reason.
Consequently, the corresponding parametric form of this constraint is
\begin{equation*}\label{p1eqp}
(a_{11}+ \triangle a_{11}\lambda )x_1+\cdots + (a_{1n}+ \triangle a_{1n}\lambda ) x_n+ x_{n+1} = b_{1}+ \triangle b_{1}\lambda ,
\end{equation*}
where $x_{n+1}$ is an slack variable. Equivalently,
\begin{equation}\label{eqchp}\nonumber
\begin{array}{rrrrl}
  a_{11} x_1 +\cdots +&a_{1n} x_n+&x_{n+1}+&\lambda z&=b_{1}\\
  \triangle a_{11} x_1 +\cdots  +&\triangle  a_{1n} x_n+&-&z&=\triangle b_{1}
\end{array}
\end{equation}
where $\lambda $ could be considered as the degree of quality.
 Without loss of generality, let $\lambda =0$ be a change point of the corresponding linear optimization problem. This means that for $\lambda =0, z=o.$ Further, let
 for some $\lambda >0, \ z=\triangle a_{11} x_1 +\cdots +\triangle  a_{1n} x_n-\triangle b_{1}>0, $ and for some $\lambda <0, \ z=\triangle a_{11} x_1  +\cdots+\triangle  a_{1n} x_n-\triangle b_{1}<0.$
 This means that increasing (decreasing)  of the quality degree $\lambda $ implies in growth (decline) of $z.$ Here, $z$ is the difference between the available variation in corresponding total production time  $(\triangle b_{1})$ and the necessary time $\triangle a_{11} x_1 +\cdots +\triangle  a_{1n} x_n$ to change the quality of products. More clearly, the amount of variation in total production time  for the current quality, $(\triangle b_{1}),$ equals to the variation of production time of every product at an optimal production plan. Increasing of the quality $(\lambda >0),$ implies in slack in total considered extra time $\triangle b_{1}$ versus the necessary time $\triangle a_{11} x_1  +\cdots +\triangle  a_{1n} x_n$ in an optimal solution at $\lambda >0$ (i.e. $z>0$). On the other hand, decreasing the quality $ (\lambda <0),$ leaves some spare time from $\triangle a_{11} x_1  +\cdots +\triangle  a_{1n} x_n$ that is more than the reduced available time $\triangle b_{1}$ (i.e. $z<0$). This information would guide the manager to adjust the production plan efficiency.

\subsection{Identifying an induced invariancy interval }
Note that Cond.s 1-3 in Theorem  \ref{thg} quids us to derive some computational tools that lead to identifying an induced invariancy interval.
Realization theory helps us to translate Property 1-3 for Problem $P_{\lambda }(\mathbf{\triangle  A})$ to computational forms in terms of eigenvalues of some matrices. In the sequel,  these conditions are translated in terms of eigenvalues of specific matrices.

Consider Properties 1-3 (See Page \pageref{proper13}). For the principal Problem $P_{\lambda _{0}}(\triangle A)$ with optimal partition $\pi (\lambda _0)=({B}(\lambda _0),{N}(\lambda _0)),$ at $\lambda _0$ the approach in \cite{GM15} needs to calculate ${A}^{\dagger}_{\tau }(\lambda ).$
To determine this matrix, we consider three possibilities. First let $m<l,$ with $l$ as defined in \eqref{taus}. In this case ${A}_{\tau }(\lambda  _{0}){A}^{\dagger}_{\tau }(\lambda  _{0})=I_{m},$
then
\begin{equation}\label{1fr}
{A}^{\dagger}_{\tau}(\lambda  ) =  ({A}_{\tau}(\lambda  _{0})+(\lambda -\lambda  _{0}){\triangle  {A}_{\tau}})^{\dagger} =(I_{l}+(\lambda  -\lambda  _{0}) {A}^{\dagger}_{\tau}(\lambda  _{0}) {\triangle  {A}_{\tau}})^{-1}  {A}^{\dagger}_{\tau}(\lambda  _{0}).
 \end{equation}
 When $m>l,$ then ${A}^{\dagger}_{\tau }(\lambda  _{0}){A}_{\tau }(\lambda  _{0})=I_{l},$ and therefore
 \begin{equation}\label{2fr}
{A}^{\dagger}_{\tau}(\lambda  ) ={A}^{\dagger}_{\tau}(\lambda  _{0})(I_{m}+(\lambda  -\lambda  _{0})  {\triangle  {A}_{\tau}}{A}^{\dagger}_{\tau}(\lambda  _{0}))^{-1}.
 \end{equation}
 Finally for $m=l,$  ${A}_{\tau}(\lambda  _{0})$ has full row and full column ranks and Moore-Penrose inverse is reduced to the standard inverse, i.e., ${A}^{\dagger}_{\tau}(\lambda  ) ={A}^{-1}_{\tau}(\lambda  ),$ then
 \begin{equation}\label{3fr}
{A}^{\dagger}_{\tau}(\lambda ) ={A}^{-1}_{\tau}(\lambda ) =(I_{m}+(\lambda -\lambda _{0}) {A}^{-1}_{\tau}(\lambda  _{0}) {\triangle  {A}_{\tau}})^{-1}  {A}^{-1}_{\tau}(\lambda  _{0}).
 \end{equation}

Note that applying \eqref{1fr} or \eqref{2fr} leads to identical nonzero eigenvalues of some matrices we need in our procedure. More clearly, using \eqref{1fr} leads to identify eigenvalues of multiplication of two matrices, say $TQ,$ while using \eqref{2fr} needs to do the same for $QT.$
Observe that \label{QR} for two matrices $Q $ and $T$ with appropriate sizes, nonzero eigenvalues of $TQ$ and $QT$ are identical, and the extra ones are zero.
Therefore, depending on the size of $TQ$ and $QT,$ more-cost-effective calculation suggests to apply the one with less dimension. In our problem in question, $l $  in the principal problem is replaced by  $\bar{l} $ and  $m $ by  $2m .$ Thus without loss or generality, we assume that $2m \leq \bar{l} $  and construct our methodology based on \eqref{1fr}.

 The following theorem provides an inequality, which is identical to Cond. 1, the existence of the pseudo-inverse of~$\mathbf{A}_{\bar{\tau }}(\lambda  )$.
\begin{theorem}\label{th1c}
For a given $\lambda _0$, let $\bar{\tau}$ correspond to the induced optimal partition $\bar{\pi }(\lambda _0)$ of Problem $P_{\lambda _{0}}(\mathbf{\triangle  A}).$ Then, for all $\lambda  \in \Lambda_{\bar{\pi}(\lambda _{0})} $ with  $\lambda \neq \lambda _{0} $,
 $\mathbf{A}_{\bar{\tau}}(\lambda  )= \mathbf{A}_{\bar{\tau}} +\lambda  \mathbf{\triangle A}_{\bar{\tau}}$  has pseudo-inverse if and only if
\begin{equation}\label{yt}
1+ \alpha_{j} (\lambda  -\lambda _{0}) \neq 0,\ j = 1,\ldots , \bar{l},
\end{equation}
where the nonzero values of $\alpha_{1},\ldots ,  \alpha_{\bar{l}}$ are nonzero eigenvalues of
$ \mathbf{\triangle  A}_{\bar{\tau}}\mathbf{A}^{\dagger}_{\bar{\tau}}(\lambda _{0}).$
\end{theorem}
\begin{proof}
For $\lambda \neq \lambda _{0} $ and  considering \eqref{1fr}, we have
\begin{equation*}\label{1frj}
\begin{array}{rcl}
\mathbf{A}^{\dagger}_{\bar{\tau} }(\lambda  ) &=&(I_{\bar{l}}+(\lambda  -\lambda  _{0}) \mathbf{A}^{\dagger}_{\bar{\tau}}(\lambda  _{0}) {\triangle  \mathbf{A}_{\bar{\tau}}})^{-1}  \mathbf{A}^{\dagger}_{\bar{\tau}}(\lambda  _{0})\\[2mm]
&=&\dfrac{1}{\lambda  -\lambda  _{0}}(\dfrac{1}{\lambda  -\lambda  _{0}} I_{\bar{l}}+ \mathbf{A}^{\dagger}_{\bar{\tau}}(\lambda  _{0}) {\triangle  \mathbf{A}_{\bar{\tau}}})^{-1}  \mathbf{A}^{\dagger}_{\bar{\tau}}(\lambda  _{0}).
\end{array}
 \end{equation*}
Since $\mathbf{A}^{\dagger}_{\bar{\tau}}(\lambda  _{0})$ exists, then $\mathbf{A}^{\dagger}_{\bar{\tau}}(\lambda  )$ exists if and only if
\begin{equation}\label{rad}
\frac{1}{\lambda _{0} - \lambda } \not \in \rho({{\mathbf{A}^{\dagger}_{\bar{\tau}}}(\lambda _{0})}{\triangle  \mathbf{A}_{\bar{\tau}}}),
\end{equation}
where for a square matrix $D$, $\rho (D) \subset \mathbb{C}$ denotes \label{2nm} the resolvent set of $D$ consisting of those complex numbers $\mu $ for which $\mu I - D $ is invertible~\cite{GM15}.
Thus, (\ref{rad}) is identical to \eqref{yt}
where the nonzero values of $\alpha_{1},\ldots ,  \alpha_{\bar{l}}$ are nonzero eigenvalues of
$ \mathbf{\triangle  A}_{\bar{\tau}}\mathbf{A}^{\dagger}_{\bar{\tau}}(\lambda _{0}).$
The proof is complete.
\end{proof}
\begin{rem} \label{rm1}
Recall that for  $2m > \bar{l},$ the size of $\mathbf{\triangle  A}_{\bar{\tau}} \mathbf{A}^{\dagger}_{\bar{\tau}}(\lambda _{0})  $ is greater than the size of  $\mathbf{A}^{\dagger}_{\bar{\tau}}(\lambda _{0}) \mathbf{\triangle  A}_{\bar{\tau}} .$ Thus, it is more cost-effective to consider
 $\alpha_{1},\ldots ,  \alpha_{\bar{l}}$ as the eigenvalues of
$\mathbf{A}^{\dagger}_{\bar{\tau}}(\lambda _{0}) \mathbf{\triangle  A}_{\bar{\tau}} .$
\end{rem}
\begin{corollar}\label{jhy}
In the special case, let $P_{\lambda _{0}}(\mathbf{\triangle  A})$ have a unique basic optimal solution, i.e., $\bar{l}=2m$ (See \eqref{3fr}). Then, for all $\lambda  \in \Lambda_{\bar{\pi}(\lambda _{0})} ,$
$\mathbf{A}_{\bar{\tau}}(\lambda  )$ is invertible if and only if~ \eqref{yt} holds, and ${\alpha }_{1},\ldots ,  {\alpha }_{2m}$ are eigenvalues of ${\mathbf{A}_{\bar{\tau}}^{-1}(\lambda _{0})} \triangle  \mathbf{A}_{\bar{\tau}}.$
\end{corollar}
The following theorem translates  Cond. 2 in terms of eigenvalues of some further specific matrices.
\begin{theorem}\label{th2c}
Let $\bar{\tau}$ correspond to the induced optimal partition $\bar{\pi}(\lambda _0)$ of Problems $P_{\lambda _{0}}(\mathbf{\triangle  A})$ and $D_{\lambda _{0}}(\mathbf{\triangle  A}).$
Then, for  $ 1\leq q \leq \bar{l}$, and $q \in {B} \cup {B^{+}}, \ ({\mathbf{x}}_{\bar{\tau}}(\lambda  ))_{q}> 0$ is identical with
   \begin{equation} \label{gesikt1}
 \displaystyle \prod_{j =1}^{\bar{l}} \frac{1+(\lambda  -\lambda _{0})\beta _{q,j}}{1+(\lambda  -\lambda _{0}) {\alpha }_j}
\left\{ \begin{array}{ll}
\geq 1, & \mbox{if}  \  (\lambda  -\lambda _{0}) \geq 0, \\
\leq 1, & \mbox{if}  \  (\lambda  -\lambda _{0}) \leq 0.\\
\end{array}
\right.
  \end{equation}
When $q \in {B^{-}}, \  ({\mathbf{x}}_{\bar{\tau}}(\lambda ))_{q}< 0 $ is identical with
  \begin{equation} \label{gesikt2}
 \displaystyle \prod_{j =1}^{\bar{l}} \frac{1+(\lambda  -\lambda _{0})\beta _{q,j}}{1+(\lambda  -\lambda _{0}) {\alpha }_j}
\left\{ \begin{array}{ll}
\leq 1, & \mbox{if}  \  (\lambda  -\lambda _{0}) \geq 0, \\
\geq 1, & \mbox{if}  \  (\lambda  -\lambda _{0}) \leq 0,\\
\end{array}
\right.
  \end{equation}
where nonzero values of $\beta _{q,1}, \ldots ,\beta _{q,\bar{l}}$ are nonzero eigenvalues of the matrix $({\triangle  \mathbf{A}}_{\bar{\tau}}+ \mathbf{b}  e^{T}_{q}) {\mathbf{A}^{\dagger}_{\bar{\tau}}(\lambda _{0})},$ and $\alpha _{1}, \ldots ,\alpha _{\bar{l}}$ are nonzero eigenvalues of
$ \mathbf{\triangle  A}_{\bar{\tau}}\mathbf{A}^{\dagger}_{\bar{\tau}}(\lambda _{0}).$
\end{theorem}
\begin{proof}
We prove \eqref{gesikt1}, the proof of \eqref{gesikt2} goes similarly, and we omit it.
 For $ 1\leq q \leq \bar{l}$ and $q \in {B} \cup {B^{+}},$ the inequality $({\mathbf{x}}_{\bar{\tau}}(\lambda  ))_{q}= e_{q}^{T} \mathbf{A}_{\bar{\tau}}^{\dagger}(\lambda  ) \mathbf{b}> 0$ holds if and only if
$1+(\lambda  -\lambda _{0}) e_{q}^{T} \mathbf{A}_{\bar{\tau}}^{\dagger}(\lambda  ) \mathbf{b}\geq 1$ when $\lambda  -\lambda _{0} \geq 0,$ and $1+(\lambda  -\lambda _{0}) e_{q}^{T} \mathbf{A}_{\bar{\tau}}^{\dagger}(\lambda  ) \mathbf{b}\leq 1$  when $\lambda  -\lambda _{0} \leq 0.$ Using \eqref{1fr}, we have
 \begin{equation}\label{Cond.2t}
 \begin{array}{l}
1+(\lambda  -\lambda _{0}) e_{q}^{T} \mathbf{A}_{\bar{\tau}}^{\dagger}(\lambda  ) \mathbf{b} \\\hspace{1cm}
=1+(\lambda  -\lambda _{0}) e_{q}^{T}(I_{\bar{l}}+(\lambda  -\lambda _{0})\mathbf{A}_{\bar{\tau}}^{\dagger}(\lambda  _{0}) {\triangle  \mathbf{A}}_{\bar{\tau}})^{-1} \mathbf{A}_{\bar{\tau}}^{\dagger}(\lambda  _{0}) \mathbf{b}.
 \end{array}
 \end{equation}
Based on the realization theory and considering
 \begin{equation*}\label{realth}
 \begin{array}{rcl}
  \lambda  &:=&\lambda  -\lambda _{0}, ~~c^{T} := {e^{T}_{q}},\\
   C&:=&\mathbf{A}^{\dagger}_{\bar{\tau}}(\lambda _{0})  \triangle  \mathbf{A}_{\bar{\tau}},~~ b:=\mathbf{A}^{\dagger}_{\bar{\tau}}(\lambda _{0}) \mathbf{b },\\
C^{\times} &:=&C+bc^{T} :=\mathbf{A}^{\dagger}_{\bar{\tau}}(\lambda _{0})({\triangle  \mathbf{A}_{\bar{\tau}}}+ \mathbf{ b}{e^{T}_{q}}),
 \end{array}
 \end{equation*}
Eq. \eqref{Cond.2t} can be converted to
  \begin{equation*} \label{gesik}
 \displaystyle \prod_{j =1}^{\bar{l}} \frac{1+(\lambda  -\lambda _{0})\beta _{q,j}}{1+(\lambda  -\lambda _{0}) {\alpha }_j},
  \end{equation*}
where  the nonzero values of $\beta _{q,1}, \ldots ,\beta _{q,\bar{l}}$ are nonzero eigenvalues of $({\triangle  \mathbf{A}}_{\bar{\tau}}+ \mathbf{b}  e^{T}_{q}) {\mathbf{A}^{\dagger}_{\bar{\tau}}(\lambda _{0})},$ and $\alpha _{1}, \ldots ,\alpha _{\bar{l}}$ are nonzero eigenvalues of
$ \mathbf{\triangle  A}_{\bar{\tau}}\mathbf{A}^{\dagger}_{\bar{\tau}}(\lambda _{0}).$
This completes the proof.
\end{proof}
\begin{rem} \label{rm2g}
In addition to Remark \ref{rm1}, observe that for $2m > \bar{l},$ the size of   $({\triangle  \mathbf{A}}_{\bar{\tau}}+ \mathbf{b}  e^{T}_{q}){\mathbf{A}^{\dagger}_{\bar{\tau}}(\lambda _{0})},$ is greater than the size of ${\mathbf{A}^{\dagger}_{\bar{\tau}}(\lambda _{0})} ({\triangle  \mathbf{A}}_{\bar{\tau}}+ \mathbf{b}  e^{T}_{q}).$ Thus, one has better to consider
$\beta _{q,1}, \ldots ,\beta _{q,\bar{l}}$ as the eigenvalues of ${\mathbf{A}^{\dagger}_{\bar{\tau}}(\lambda _{0})} ({\triangle  \mathbf{A}}_{\bar{\tau}}+ \mathbf{b}  e^{T}_{q})$.
\end{rem}

Analogous to Theorem \ref{th2c}, next theorem translates  Cond. 3   in terms of eigenvalues of some other specific matrices.
{ Recall that in this theorem functions $\tau $ and ${\tau }'$ are defined as \eqref{taus} for Problem $P({\triangle  A},{\triangle  b})$.  }
\begin{theorem}\label{then}
Let $\bar{\pi }(\lambda  _{0}) $ be the induced optimal partition of Problems $P_{\lambda _{0}}(\mathbf{\triangle  A})$ and $D_{\lambda _{0}}(\mathbf{\triangle  A})$.
 For $p \in {\rm Range}({{\tau  }'}), $ the inequality $c_{p} -c_{\tau }^{T} A^{\dagger}_{\tau} (\lambda ) A_{p}(\lambda  ) >0$
 is translated as
 \begin{equation}\nonumber
\displaystyle \prod_{j=1}^{l} \frac{1+(\lambda  -\lambda _{0}) \gamma _{p,j}}{1+(\lambda  -\lambda _{0}) {\alpha ^{'}}_j} +\frac{1}{\lambda  -\lambda _{0}} \displaystyle \prod_{j=1}^{l} \frac{1+(\lambda  -\lambda _{0}) \delta _{p,j}}{1+(\lambda  -\lambda _{0}) \alpha ^{'}_j} < 1 + \frac{1}{\lambda  -\lambda _{0}} +{c}_{p},
 \end{equation}
where  the nonzero values of ${\alpha ^{'}_{1}}, \ldots ,{\alpha ^{'}_{l}},$
$\gamma _{p,1}, \ldots ,\gamma _{p,l}$
and $ \delta _{p,1}, \ldots ,\delta _{p,l}$ are nonzero eigenvalues of the matrices  $\triangle  A_{\tau} A^{\dagger}_{\tau}(\lambda _{0}), $  $ (\triangle  A_{\tau}+\triangle  A_{p}{c}_{\tau}^T) A^{\dagger}_{\tau }(\lambda _{0}),$ and $(\triangle  A_{\tau}+(A_{p}+\lambda _{0} \triangle  A_{p}) c_{\tau }^T ) A^{\dagger}_{\tau }(\lambda _{0}),$ respectively.
\end{theorem}
\begin{proof}
Consider the dual constraints of  $ D_{\lambda } (\mathbf{\triangle  A})$ as
  \begin{equation}\nonumber
\begin{pmatrix}
A^{T} & \triangle  A^{T}  \\
 \lambda  I_{m}& -I_{m}
\end{pmatrix} \begin{pmatrix}
y \\
 w
\end{pmatrix}+\begin{pmatrix}
s \\
0
\end{pmatrix}=\begin{pmatrix}
c \\
0
\end{pmatrix}.
\end{equation}
Observe that  the vector $0$ in
$\mathbf{s}=\begin{pmatrix} s\\ 0 \end{pmatrix} $
is of dimension $m,$ and $s=c-A^{T}y-\triangle  A^{T}  w .$ Thus, $\mathbf{s}\geq 0$ is identical with  ${s}\geq 0.$ Further,  $\mathbf{s}\geq 0$ is the feasibility of $ D_{\lambda } (\mathbf{\triangle  A})$ which is identical with the optimality of $ P_{\lambda _{0}} (\mathbf{\triangle  A}).$ Thus,   one can replace ${\bar{\tau }}^{'}$ with ${\tau }'$ in Cond.~3.
On the other hand,  $\tau $ and ${\tau }'$  denote the sets of indices in $\mathbf{x}$  respectively corresponding to  positive and zero variables of $x$.
 Thus,  $c_{\tau }^{T} A^{\dagger}_{\tau} (\lambda  ) A_{p}(\lambda  ) - c_{p} <0$  can be rewritten as $$ 1+ \frac{1}{\lambda  -\lambda _{0}} + c_{\tau }^{T} A^{\dagger}_{\tau} (\lambda ) (A_{p}+(\lambda -\lambda _{0}) \triangle  A_{p} +\lambda _{0}\triangle  A_{p}) <1+ \frac{1}{\lambda  -\lambda _{0}} + c_{p},$$
 for $p \in {\rm Range}({ {\tau  }'})$. Equivalently,  as
\begin{equation}\label{uh}
\begin{array}{c}
 1+(\lambda -\lambda _{0}) c_{\tau }^{T} A^{\dagger}_{\tau} (\lambda  ) \triangle  A_{p} + \frac{1}{\lambda  - \lambda _{0}}( 1+(\lambda -\lambda _{0}) c_{\tau }^{T} A^{\dagger}_{\tau} (\lambda  )(A_{p}+\lambda _{0}\triangle  A_{p})) \\ [2mm]<1+ \frac{1}{\lambda  -\lambda _{0}} + c_{p}.
 \end{array}
\end{equation}
 Considering \eqref{1fr} and realization theory, \eqref{uh} can be reworded as
 \begin{equation}\nonumber
\displaystyle \prod_{j=1}^{l} \frac{1+(\lambda  -\lambda _{0}) \gamma _{p,j}}{1+(\lambda  -\lambda _{0}) {\alpha ^{'}}_j} +\frac{1}{\lambda  -\lambda _{0}} \displaystyle \prod_{j=1}^{l} \frac{1+(\lambda  -\lambda _{0}) \delta _{p,j}}{1+(\lambda  -\lambda _{0}) \alpha ^{'}_j} < 1 + \frac{1}{\lambda  -\lambda _{0}} +{c}_{p},
 \end{equation}
where ${\alpha ^{'}_{j}},\gamma _{p,j}$ and $ \delta _{p,j},\  j=(1,\ldots ,l)$ are as stated. The proof is complete.
\end{proof}

\begin{rem} \label{rm3g}
For  $m>l$, one has better to consider
${\alpha ^{'}_{j}},\gamma _{p,j}$ and $ \delta _{p,j},\ j=(1, \ldots ,l)$ as the eigenvalues of $ A^{\dagger}_{\tau}(\lambda _{0}) \triangle  A_{\tau}, A^{\dagger}_{\tau }(\lambda _{0}) (\triangle  A_{\tau}+\triangle  A_{p}{c}_{\tau}^T) $ and $A^{\dagger}_{\tau }(\lambda _{0}) (\triangle  A_{\tau}+(A_{p}+\lambda _{0} \triangle  A_{p}) c_{\tau }^T )$, respectively.
\end{rem}

\section{Closed form of the optimal value function}
Without loss of generality, let $\bar{l} > 2m.$ In the sequel,
 the representation of the optimal value function $Z(\lambda  )$ is derived.
\begin{theorem}
Let  Cond.s 1-3  satisfy for a  fixed $\lambda  \in \Lambda $.
The representation of the optimal value function $Z(\lambda )$ is
\begin{equation}\label{obcr}
 Z(\lambda )=\frac{1}{\lambda  -\lambda _{0}} \displaystyle \prod_{j=1}^{\bar{l}} \frac{1+(\lambda  -\lambda _{0}) {{\alpha}} ^{ \times }_{j}}{1+(\lambda -\lambda _{0}) {{\alpha}}_j}-1,
 \end{equation}
 where the nonzero values  ${{\alpha}} ^{ \times }_{1}, \ldots ,{{\alpha}}^{ \times }_{\bar{l}}$ are nonzero eigenvalues of the matrix  $ ({\triangle  \mathbf{A}_{\bar{\tau} }}+\mathbf{b}{\mathbf{c}}_{\bar{\tau} }^T)\mathbf{A}^{\dagger}_{\bar{\tau} }(\lambda _{0})$ and nonzero values of ${{\alpha}}_{1}, \ldots ,{{\alpha}}_{\bar{l}}$ are nonzero eigenvalues of  the matrix $ \mathbf{\triangle  A}_{\bar{\tau}}\mathbf{A}^{\dagger}_{\bar{\tau}}(\lambda _{0}).$ .
\end{theorem}
\begin{proof}
 It is clear that $Z(\lambda )= {\mathbf{c}}^{T}_{\bar{\tau} }\mathbf{x}_{\bar{\tau}}(\lambda  ) ={\mathbf{c}}^{T}_{\bar{\tau} } \mathbf{A}_{\bar{\tau} }^\dagger (\lambda  ) \mathbf{b}.$ Considering \eqref{1fr} and realization theory, it holds
\begin{equation*}
\begin{array}{rcl}
 1+(\lambda  -\lambda _{0}) Z(\lambda  )&=&1+(\lambda  -\lambda _{0}){\mathbf{c}}^{T}_{\bar{\tau} } (I_{\bar{l}}+(\lambda  -\lambda  _{0}) \mathbf{A}^{\dagger}_{\bar{\tau}}(\lambda  _{0}) {\triangle  \mathbf{A}_{\bar{\tau}}})^{-1}  \mathbf{A}^{\dagger}_{\bar{\tau}}(\lambda  _{0})\mathbf{b}\\[2mm]
&=& \displaystyle \prod_{j=1}^{\bar{l}} \frac{1+(\lambda  -\lambda _{0}) {{\alpha}} ^{ \times }_{j}}{1+(\lambda  -\lambda _{0}) {{\alpha}}_j}.
\end{array}
 \end{equation*}
 Thus,
 \begin{equation*}
 Z(\lambda)=\frac{1}{\lambda  -\lambda _{0}} \displaystyle \prod_{j=1}^{\bar{l}} \frac{1+(\lambda  -\lambda _{0}) {{\alpha}} ^{ \times }_{j}}{1+(\lambda  -\lambda _{0}) {{\alpha}}_j}-1
 \end{equation*}
  where ${{\alpha}} ^{ \times }_{1}, \ldots ,{{\alpha}}^{ \times }_{\bar{l}}$ and ${{\alpha}}_{1}, \ldots ,{{\alpha}}_{\bar{l}}$ are as stated.
\end{proof}
\begin{corollar}
In the spacial case $\bar{l}=2m$, the  closed form of the optimal value function is
$$Z(\lambda )=\frac{1}{\lambda  -\lambda _{0}}  \displaystyle \prod_{j=1}^{2m} \frac{1+(\lambda  -\lambda _{0}) {{\alpha}} ^{ \times }_{j}}{1+(\lambda  -\lambda _{0}) {{\alpha}}_j}- 1,$$
  where ${{\alpha}} ^{ \times }_{1}, \ldots ,{{\alpha}}^{ \times }_{2m}$ are eigenvalues of  $ \mathbf{A}^{-1}_{\bar{\tau} }(\lambda _{0})({\triangle  \mathbf{A}_{\bar{\tau} }}+\mathbf{b}{\mathbf{c}}_{\bar{\tau} }^T)$ and ${{\alpha}}_{1}, \ldots ,{{\alpha}}_{2m}$ are eigenvalues of ${\mathbf{A}_{\bar{\tau}}^{-1}(\lambda _{0})} \triangle  \mathbf{A}_{\bar{\tau}}.$
\end{corollar}
\begin{rem}
Similar to Remarks \ref{rm1}-\ref{rm3g},
for $2m > \bar{l},$ one has better to consider
 ${{\alpha}} ^{ \times }_{1}, \ldots ,{{\alpha}}^{ \times }_{\bar{l}}$ and $\alpha_{1},\ldots ,  \alpha_{\bar{l}}$ as the eigenvalues of  the matrices $ \mathbf{A}^{\dagger}_{\bar{\tau} }(\lambda _{0}) ({\triangle  \mathbf{A}_{\bar{\tau} }}+\mathbf{b}{\mathbf{c}}_{\bar{\tau} }^T) $ and $\mathbf{A}^{\dagger}_{\bar{\tau}}(\lambda _{0}) \mathbf{\triangle  A}_{\bar{\tau}}$, respectively.
\end{rem}

\section{Finding all transition and change points }\label{sec6}
The following computational algorithm is devised to find all transition and change points of Problem  $ P_{\lambda } (\mathbf{\triangle  A})$ in $\mathbb{R}^{+}$. All induced invariancy intervals are identified by these points. The inputs of this algorithm is the perturbing direction $({\triangle  A},{\triangle  b})$ in addition to the fixed data $A,b$ and $c.$ After finding a transition point or a change point, the algorithm needs to select a point in the next immediate right invariancy interval as close as possible to the aforementioned identified point. This selection could be carried out by adding $\epsilon >0 $ to the identified transition or change point. The output of this algorithm is the point  $\bar{\lambda  }_{i} $ and $\underline{\lambda  }_{i}$ while it is expected that $\bar{\lambda  }_{i})=\underline{\lambda  }_{i+1}.$ These points are transition or change points when the associated problems in these points have optimal solutions. Distinguishing between transition points and change points can be carried out using the induced optimal partitions at their immediate adjacent intervals. The algorithm consists of the following steps.
\begin{description}
\item[\textbf{Step 0}:]
Let $i=0, \lambda _{i}=0$.
\item[\textbf{Step 1}:]
Solve Problems $P_{\lambda _{i}}(\mathbf{\triangle  A})$ and $P_{\lambda _{i}+\epsilon}(\mathbf{\triangle  A}),$ and identify $\bar{\pi }(\lambda  _{i})$ and $\bar{\pi }(\lambda  _{i}+ \epsilon)$, respectively. If $P_{\lambda _{i}+\epsilon}(\mathbf{\triangle  A})$ is infeasible or unbounded then stop. $ \lambda _{i}$ is a transition or change point
\item[\textbf{Step 2}:]
If $\bar{\pi }(\lambda  _{i}) \neq \bar{\pi }(\lambda  _{i}+\epsilon),$ then $\lambda  _{i}$ is a transition or change point.Let let $\underline{\lambda  }_{i} =\bar{\lambda  }_{i}=0,\  i=i+1,\ \lambda _{i} :=\bar{\lambda }_{i} +\epsilon .$
\item[\textbf{Step 3}:]
Find interval $(\underline{\lambda  }_{i} , \bar{\lambda  }_{i})$ as the intersection of the obtained intervals in Theorems \ref{th1c},  \ref{th2c}, and  \ref{then} for Problem  $P_{\lambda _{i}}(\mathbf{\triangle  A})$.
\item[\textbf{Step 4}:]
If $\bar{\lambda }_{i} =\infty$ stop, otherwise solve $P_{\bar{\lambda }_{i}}(\mathbf{\triangle  A}).$ If this problem is unbounded or infeasible, stop. Otherwise, identify $\bar{\pi }(\bar{\lambda } _{i}).$ Set $\lambda _{i+1} :=\bar{\lambda}_{i} +\epsilon , i:=i+1 $
\item[\textbf{Step 5}:]
Solve Problem $P_{\lambda _{i}}(\mathbf{\triangle  A})$. If this problem is infeasible or unbounded, then stop. Otherwise, identify $\bar{\pi }(\lambda  _{i})$ and go to \textbf{Step 3}.
\end{description}

\begin{rem} To find the induced optimal invariancy intervals and transition or change points to the left of $\lambda _{0}=0$, one can replace $\epsilon <0$  and substitutes $\bar{\lambda}$ with $\underline{\lambda }$ in Steps 2 and 4 of the algorithm.  This algorithm terminates in a finite number of iterations since the number of induced optimal partitions is finite.
\end{rem}
\section{Illustrative Examples}\label{sec7}
In this section, two concrete examples  are designed to clarify the approach.
In these examples, we set $\ \epsilon =0.1$.
The first example is an instance that includes a point which is both transition and change point.
 \begin{exam}\label{exampl1}
\rm {Consider the following problem
\begin{equation}\label{exam1}\nonumber
\begin{array}{ll}
 \min & \  -x_1-x_2 \\
 s.t.&(1+\lambda  )x_1+(1-\lambda  )x_2+x_{3}=1+\lambda \\
 &x_1, x_2, x_3 \geq 0,
\end{array}
\end{equation}
 where
  $t \in {\mathbb R}$ is the parameter.  This problem can be rewritten as
\begin{equation}\label{exam2}\nonumber
\begin{array}{ll}
\min & \  -x_1-x_2\\
 s.t.&x_1+x_2+x_{3}+\lambda (x_1-x_2-1)=1\\
 &x_1, x_2, x_3  \geq 0.
\end{array}
\end{equation}
By considering $x_{4}=x_1-x_2-1,$ this problem is equivalent with
 \begin{equation}\label{ex1npiex}
\begin{array}{lrrlrl}
 \min & \  -x_1&-x_2&& \\
 s.t.&x_1&+x_2&+x_{3}&+\lambda  x_{4}&=1\\
 &x_{1}&-x_{2}&&-x_{4}&=1\\
 &x_1,& x_2,& x_3, \geq 0,&  &
\end{array}
\end{equation}
in which only the coefficient matrix is perturbed. To be clear
\begin{equation}\nonumber
\mathbf{c}=\begin{pmatrix}
-1 \\-1\\ 0\\
0
\end{pmatrix},
\mathbf{A}=\begin{pmatrix}
1 &1& 1 & 0  \\
1& -1& 0 & -1
\end{pmatrix},
  \mathbf{\triangle  A}= \begin{pmatrix}
0&0& 0& 1\\
0&0& 0 & 0
\end{pmatrix},
  \mathbf{b}=\begin{pmatrix}
1 \\
1
\end{pmatrix}.
\end{equation}
 In the first iteration of the algorithm, when $\lambda  _{0}= 0$ and $ \lambda  _{1}= 0.1$, implementation of Step 1 reveals the induced primal-dual optimal solution of Problem \eqref{ex1npiex} as
 $$\mathbf{x}^*(0)=(0.56, 0.44, 0, -0.88), \mathbf{y}^*(0)=(-1,0), \mathbf{s}^*(0)=(0, 0, 1, 0)$$
 $$\mathbf{x}^*(0.1)=(0, 1.22, 0, -2.22), \mathbf{y}^*(0.1)=(-1.11,-0.11), \mathbf{s}^*(0.1)=(0.22,0, 1.11, 0).$$
  Corresponding induced optimal partitions are
 $$B(0)=\{ 1,2 \}, {B^{+}(0)}=\varnothing , {B^{-}(0)}=\{ 4\}, {N(0)}=\{3\}, {N^{0}(0)}=\varnothing,$$
 $${B(0.1)}=\{ 2 \}, {B^+(0.1)}=\varnothing , {B^-(0.1)}=\{ 4\}, {N(0.1)}=\{1,3\}, {N^0(0.1)}=\varnothing.$$
 By  Step 2, $\lambda  _{0}=0$ is as a transition or change point.
  From Step 3,  the corresponding  invariancy interval containing $\lambda _0=0.1$, is $(0,1)$, and $\bar{\lambda  }_{1}=1.$ Since  Problem \eqref{ex1npiex} at this point is  unbounded, the algorithm is terminated at step 4.

  This algorithm can be applied to find the invariancy intervals to the left of $\lambda _0=0$.  The immediate  right invariancy interval is (-1,0), while Problem \eqref{ex1npiex} at $\lambda =-1$ is unbounded.

  The results are reported in Table \ref{jex11} which includes  the induced optimal partition invariancy intervals and corresponding  induced optimal partitions. Representation of the optimal value function on these intervals are identified using \eqref{obcr} and appeared at the last column.
   Fig. \ref{labelFig1} depicts the optimal value function on these intervals.
\begin{table}[!ht]
  \centering
\begin{tabular}{|c|c|c|c|c|c|c|}
  \hline
     ${\rm Ind. \ Inv.\ Int.}$ & ${B}$ &   ${B^{+}}$ &  ${B^{-}}$ & ${N}$ & ${N^{0}}$ & $Z(\lambda  )$\\  \hline
       $(-1,0)$  & $\{ 1\}$ & $\varnothing $ & $\varnothing $ & $\{2,3\}$ & $\{ 4\}$ & $-1$ \\            \hline
     $ 0 $ &  $\{1,2\}$ & $\varnothing $ &  $\{ 4\}$ & $\{ 3\}$ & $\varnothing $ & $ -1 $ \\            \hline
             $(0,1)$  &  $\{ 2\}$ &  $\varnothing $ &  $\{ 4\}$ &  $\{ 1,3\}$  & $\varnothing $  & $\displaystyle \frac{\lambda  +1}{\lambda  -1}$ \\   \hline
\end{tabular}
  \caption{ Induced invariancy intervals, induced optimal partitions and  optimal value function in Example \ref{exampl1}.} \label{jex11}
\end{table}

As reflected in  Table \ref{jex11}, domain of the optimal value function is the open interval  $(-1,1).$
 It is continuous at point $\lambda =0$,  but fails to be continuous  or even defined at the end points.
 Observe that  when one passes from $\lambda =0$ to each neighboring interval, an index from $B$ moves to $N.$ Thus $\lambda =0$ is a transition point. Further, moving from the interval $(-1,0)$ to 0, not only $B$ and $N$ exchanges the index 2, but also ${B^{-}}$ and ${N^{0}}$ exchanges the index 4. Thus $\lambda =0$ is a change point, too.
\begin{figure}[!ht]
\begin{center}
\includegraphics[width=1\textwidth]{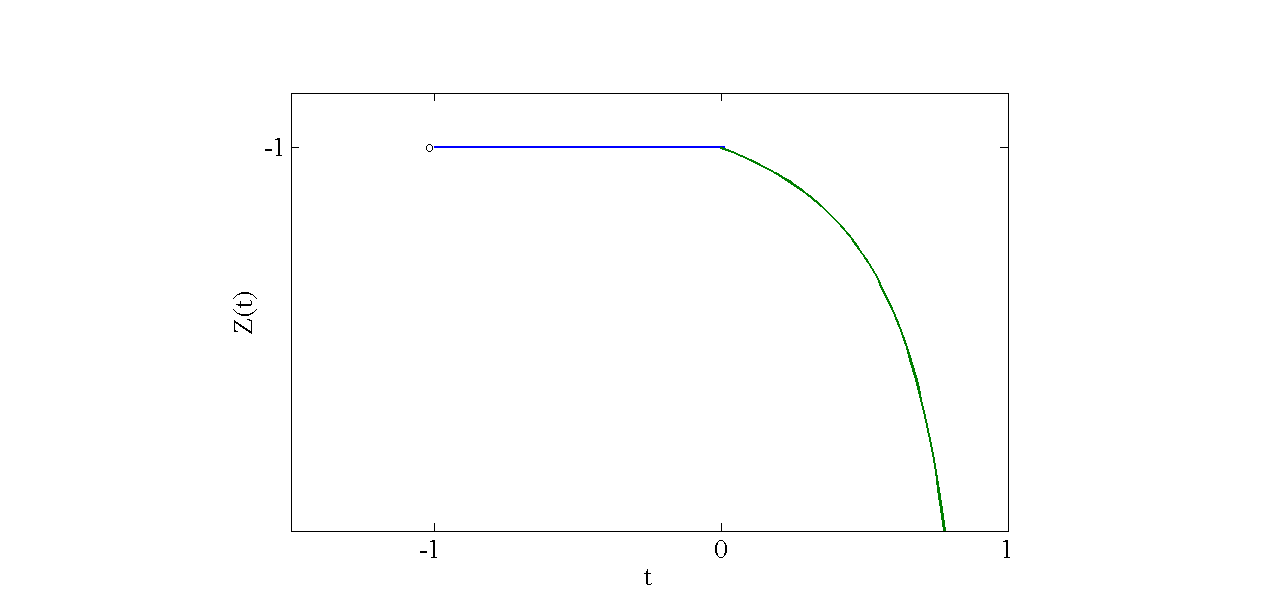}
\caption {The optimal value function $Z(\lambda )$ of Example \ref{exampl1}.}
\label{labelFig1}
\end{center}
\end{figure}}
\end{exam}
The following example contains a change point where in its neighborhoods, the representation of the optimal value function  dose not change.
\begin{exam}\label{exampl2}
 \rm { Assume the following problem
\begin{equation}\label{exam22}\nonumber
\begin{array}{rrrrl}
 \min & \  -x_1-x_2&& \\
 s.t.&t x_1+(1+\lambda  )x_2&+x_{3}&&=1+2  \lambda \\
 &(1 - \lambda  ) x_1+(1-2 \lambda  )x_2& &+x_{4}&=1-  \lambda \\
 &x_1, x_2, x_3 ,  x_4\geq 0,& &&
\end{array}
\end{equation}
where $t \in {\mathbb R}$ is the parameter.  One can rewrite this problem as
\begin{equation}\label{exam222}\nonumber
\begin{array}{lrrrrl}
 \min & \  -x_1&-x_2& \\
 s.t.&&+x_2&+x_{3}&&+\lambda ( x_1 +x_2-2)=1\\
 & x_1&+x_2 &&+x_{4}&+\lambda (-x_1-2x_2+1)=1\\
 &x_1,& x_2,& x_3, & x_4 & \geq 0.
\end{array}
\end{equation}
Substitution of  $x_{5}=x_1 +x_2-2$ and $x_{6}=-x_1-2x_2+1,$ results to
 \begin{equation}\label{ex1np2d}\nonumber
\begin{array}{lrrrlrrl}
 \min & \  -x_1&-x_2&&&& \\
 s.t.&&x_2&+x_{3}&&+\lambda  x_{5}&&=1\\
 &x_{1}&+x_{2}&&+x_{4}&&+\lambda  x_{6}&=1\\
 &x_{1}&+x_{2}&&&- x_{5}&&=2\\
 &-x_{1}&-2x_{2}&&&&- x_{6}&=-1\\
 &x_1,& x_2,& x_3, & x_{4}\geq 0 ,&&&
\end{array}
\end{equation}
 To be clear
\begin{equation}\nonumber
\mathbf{c}=\begin{pmatrix}
-1 \\-1\\ 0\\0\\0\\0
\end{pmatrix},
\mathbf{A}=\begin{pmatrix}
0&1 &1& 0 & 0 & 0 \\
1& 1& 0 & 1 & 0 & 0 \\
1& 1& 0 & 0 & -1 & 0 \\
-1& -2& 0 & 0 & 0 & -1
\end{pmatrix},
\end{equation}
\begin{equation}\nonumber
  \mathbf{\triangle  A}= \begin{pmatrix}
0&0& 0&0& 1&0\\
0&0& 0 & 0&0&1\\
0&0& 0 & 0&0&0\\
0&0& 0 & 0&0&0
\end{pmatrix},
  \mathbf{b}=\begin{pmatrix}
1 \\
1\\
2\\
-1
\end{pmatrix}.
\end{equation}

Table \ref{j22} has a summary of the results and Fig. \ref{labelFig2} denotes the corresponding optimal value function.
 As this table shows,  $\lambda  =0$ is simultaneously transition and change point, while $\lambda  =-1$ and $\lambda  =1$ are only transition points.
 The optimal value function is continuous on them but not differentiable. The point  $\lambda  =0.5$  is a  change point, and the optimal value function is continuous and differentiable at this point.
 The domain of the optimal value function is the closed interval $[-1, \infty )$ in this example.
Here, explicitly, at $\lambda  =0,$ the indices $2,3$ interchange between $B$ and $N,$ and the index $6$  interchanges between $B^{-}$ and $N^0.$ At $\lambda =1,$ the indices $2,4$ interchange only between $B$ and $N.$ Approaching from the left to $\lambda =0.5,$ the index 5 moves from $B^-$ to $N^0,$ and passing from this point it moves to $B^+.$
As can be seen in the neighborhood of this point, the representation of the optimal value function does not alter.
This is a visible display of the difference between a mere transition point and a mere change point.
 \begin{table}[!ht]
  \centering
\begin{tabular}{|c|c|c|c|c|c|c|}
  \hline
    ${\rm Ind. \ Inv.\ Int.}$ & ${B}$ &  ${B}^{+}$ &   ${B}^{-}$ &   ${N}$ & ${N^{0}}$ & $Z(\lambda  )$  \\
    \hline
     $-1$ & $\{ 1\}$ &   $\varnothing $ &   $\{5\} $ &   $\{2,3,4\}$ & $\{6\}$  &$-1$ \\
  \hline
  $(-1,0)$ & $\{1, 3\}$ &   $\varnothing $ &   $\{5\} $ &   $\{2,4\}$ & $\{6\}$  & $-1$ \\
  \hline
  $0$ & $\{1,2, 3\}$ &   $\varnothing $ &   $\{5,6\} $ &   $\{4\}$ & $\varnothing $ & $-1$  \\
  \hline
   $(0,0.5)$ &    $\{ 1,2\}$ &  $\varnothing $ & $\{ 5,6\}$ &  $\{3,4\}$ & $\varnothing $ & $\displaystyle \frac{-1- 2 \lambda ^{2}}{\lambda  ^{2} - \lambda +1}$ \\
            \hline
            $0.5$ &      $\{1,2\}$ &  $\varnothing $ &  $\{6\}$ &  $\{3,4\}$&  $\{5\}$  & $-2$ \\
            \hline
            $(0.5,1)$ &      $\{1,2\}$ & $\{5\} $ &  $\{6\}$ &  $\{3,4\}$&  $\varnothing $ & $\displaystyle \frac{-1- 2 \lambda  ^{2}}{\lambda ^{2} -\lambda  +1}$  \\
            \hline
            $1$ &      $\{1\}$ &  $\{5\}$ &  $\{6\}$ &  $\{2,3,4\}$&  $\varnothing $   & $-3$ \\
            \hline
            $(1,\infty )$ &      $\{1,4\}$ &  $\{5\}$ &  $\{6\}$ &  $\{2,3\}$&  $\varnothing $  &  $\displaystyle \frac{-1- 2 \lambda  }{\lambda  }$\\
            \hline
\end{tabular}
  \caption{Invariancy intervals, induced optimal partitions and the optimal value function in Example \ref{exampl2}.}  \label{j22}
\end{table}

\begin{figure}[!ht]
\begin{center}
\includegraphics[width=1\textwidth]{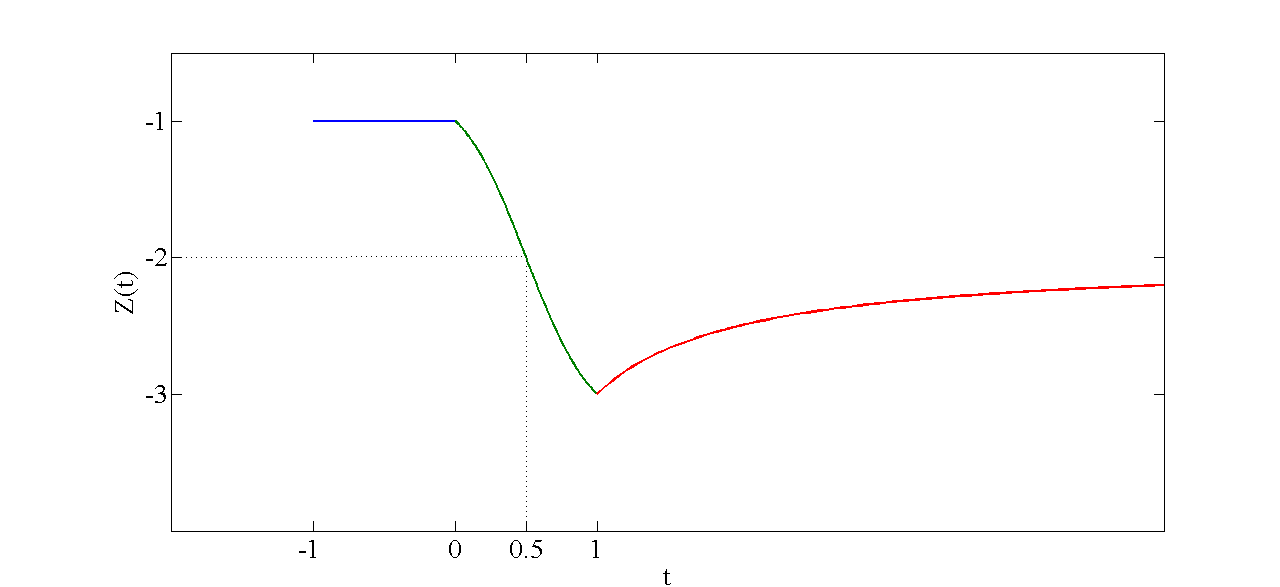}
\caption {The optimal value function $Z(\lambda  )$ in Example \ref{exampl2}.}
\label{labelFig2}
\end{center}
\end{figure}}
\end{exam}

\subsection{Computational results }
In the sequel, the results of executing the methodology on some test problems from Netlib are reported. They are in standard form, and their characteristics are reflected in Table \ref{tnetlib}.
\begin{table}[ht]
  \centering
\begin{tabular}{|c|c|c|}
 \hline
 Name &Rows &Columns\\
\hline
   Afiro &27&  51\\
   Blend  & 74 &114 \\
    Stocfor1& 117  &165 \\
   Scagr7& 129 & 185\\
  \hline
\end{tabular}
  \caption{ Characteristics of test problems} \label{tnetlib}
\end{table}

Computations are carried out on Asus FX570UD - FY213, Intel core i7-8550u with  12GB Ram with platform Windows 10 Enterprise. The algorithm has been implemented in MATLAB R2019b  using some main standard commands as linprog with the option 'interior-point-legacy', eig, pinv. We used standard commands as infsup, in0, intersect, of version 11 of the interval arithmetic toolbox, INTLAB, for interval computations.

In each parametric Problem $P(\triangle A, \triangle b)$,  $\lfloor \dfrac{n}{2} \rfloor$  elements of  $\triangle A$ are randomly selected and their values are produced by the pseduorandom normal codes of the Matlab. Further,  each element of  $\triangle b$ is produced by the uniform distribution from interval $[0,3]$.
 Computational results with $ \epsilon =0.015$ are reported in Table~\ref{j11}. 

\begin{sidewaystable}[h!]
  \centering
\begin{tabular}{|c|c|c|c|c|c|r|r|r|r|}
 \hline
 & \multirow{ 2}{*}{ Convex section }& \multicolumn{4}{c|}{No. of detected Ind.}  & \multicolumn{4}{c|}{CPU time (Sec.)}\\ \cline{3-10}
  & & Inv. &  Trans.& Chang. &  & &  &  & \\
  Problem  &of  $ \Lambda ; 0 \in \Lambda $&  Int. &  Points& Points & Both&  Min & Mean  &  Max  & Total\\
  \hline
       &  & &&&& 0.013 &   0.016 &   0.021& \\
   Afiro & (-3.136, 59.944)&  44 &29 & 3 &11 & 0.240  &  0.585  &  0.829&4151\\
  & & &  && &37.205&   64.131 &  92.316& \\
   & & & &&& 18.762&   29.608&   36.888 & \\
  \hline
      &  & &&& &0.351 &   0.376 &   0.545& \\
 Blend   &(-0.060,2.092]  & 34 &30& 3&0 &1.084&   1.384  &  2.665 &18935\\
  & & & && & 287.569&  428.632&  552.295&  \\
   & & && & &94.105&  126.503&  166.173&  \\
  \hline
       & & &&&& 2.379  &  2.600  &  7.238 & \\
   Scagr7  & [-0.466,0.668] & 35 & 31 & 0 & 3 &1.795   & 2.263  &  5.988 &70189\\
  & & & && & 959.6  &  1648.3  &  2687.8 &  \\
   & & &&&  &214.869 &  352.202 &  882.132&  \\
  \hline
\end{tabular}
  \caption{ Computational results for  $ \epsilon =0.015.$} \label{j11}
\end{sidewaystable}

In this table, the first column is the name of problems, the second column is the corresponding convex section of the domain of the optimal value function containing $\lambda _{0}=0.$ The four consequent columns are the number of detected invariancy intervals, transition points, change points and those that are both transition and change points.
The values in the last column (in CPU time) corresponds to the total times for determining all invarian????
 The ``min'' is the minimum  time spent in detection of one of the invariancy intervals. The ``mean'' is the average consumed time in detection of an invariancy interval.
 
 The ``min''  (``max'') is the minimum (maximum) time spent in detection of all invariancy intervals. The ``mean'' is the average consumed time in detection of an invariancy interval.
 The first row is related to the required time to find the eigenvalues of the matrices $ \mathbf{\triangle  A}_{\bar{\tau}}\mathbf{A}^{\dagger}_{\bar{\tau}}(\lambda _{0})$  (See Theorem \ref{th1c}),
$({\triangle  \mathbf{A}}_{\bar{\tau}}+ \mathbf{b}  e^{T}_{q}) {\mathbf{A}^{\dagger}_{\bar{\tau}}(\lambda _{0})}$ (See Theorem \ref{th2c}), and
$\triangle  A_{\tau} A^{\dagger}_{\tau}(\lambda _{0}), $  $ (\triangle  A_{\tau}+\triangle  A_{p}{c}_{\tau}^T) A^{\dagger}_{\tau }(\lambda _{0}),$ and $(\triangle  A_{\tau}+(A_{p}+\lambda _{0} \triangle  A_{p}) c_{\tau }^T ) A^{\dagger}_{\tau }(\lambda _{0})$ (See Theorem \ref{then}). The second to fourth rows are the times consumed to determine the intervals satisfying Cond.s 1-3, respectively.

 For example, in Afiro,  induced invariancy intervals are 44 with  29 points as mere transition points, 3 points as mere change points, and  11 points as both transition and change points.
 In the last column, 4151  is the total time  for determining all invariancy intervals.
The values 0.013, 0.016, and  0.021 are respectively, related to min, mean, and max of the required time to find the eigenvalues of some specific matrices in Theorems \ref{th1c}, \ref{th2c}, and \ref{then}. The values (0.240 , 0.585 , 0.829), (37.205, 64.131,  92.316), and (18.762, 29.608, 36.888) are respectively, related to min, mean, and max of the times consumed to determine the intervals satisfying Cond.s 1-3, respectively.

As  Tables \ref{tnetlib}  and  \ref{j11} show,  the computation time increases as the size of the problem increases. However, this is not a general rule. For example, the problem Stocfor1 only revealed one invariancy interval $(-0.011,0.012)$ for several random selections of $\triangle  A$ and $\triangle b$; and the total CPU time was almost less than 700 Sec.Furthermore, these results show that the time required to determine an interval satisfying Cond. 2, is as least as the two-thirds of the total time.

\begin{table}[ht]
  \centering
\begin{tabular}{|c|c|c|c|c|c|c|}
 \hline
 &\multirow{ 2}{*}{ Convex section } & \multicolumn{4}{c|}{No. of detected Ind.}  & \multirow{ 2}{*}{Total CPU  } \\ \cline{3-6}
   & &Inv. &  Trans.&  Chang. &  & \\
 Precision & of  $ \Lambda ; 0 \in \Lambda $ &  Int. &  Points & Points & Both &Time (Sec.) \\
  \hline
  0.005 & $ (-3.136,59.944)$ &56 & 36 & 7&12 &5408\\
0.010  & $ (-3.136,59.944)$ & 48  &31 &5 & 11& 5262\\
0.015  & $ (-3.136,59.944)$ & 44  &29 &3 &11 &4151 \\
0.020 & $ (-3.136,60.102)$ & 41  & 27&3 &10&3783\\
   0.025& $ (-3.136,60.102)$ & 42  &28&  2&11& 3859\\
 0.030 &  $ (-3.136,60.102)$ & 40 &25  &2& 12 & 3601\\
  0.035  &  $ (-3.136,60.102)$ & 40 &25&2  &12 &3683\\
     0.040  &  $ (-3.136,60.102)$ & 34&22 & 1&10 &3662\\
 0.045    &   $ (-3.136,60.102)$ &32 &20& 1&10& 3221\\
  0.050&  $ (-3.284,59.944)$ & 29& 17&1& 11 & 2594 \\
  \hline
\end{tabular}
  \caption{ Afiro's computational results with different accuracy values.} \label{j113}
\end{table}

In order to analyse the effect of $\epsilon $ {on the number of detected invariancy intervals}, the behavior of Afiro is investigated and the results are denoted in Table \ref{j113} for $\epsilon =0.005$ up to $\epsilon =0.050$ with step size $0.050$.
 As it is depicted in this table, when $ \epsilon $ increases,  the total running time for computing  the convex domain containing zero  would decrease.
 For example, the total running time for computing the convex invariancy interval $\Lambda $  for $ \epsilon =0.015$ and $ \epsilon =0.020$ is equal to 4151 and 3783, respectively, with  the exception at $\epsilon =0.020$ and $\epsilon =0.025$.
Analogously,  as $ \epsilon $ increases, the number of detected invariancy intervals { might decrease}. This means that some invariancy intervals might be divided into some subintervals by decreasing the value of $\epsilon$. 
 For example, the number of invariancy intervals of Problem Afiro for $ \epsilon =0.015$ and $ \epsilon =0.020$ are 44 and 41, respectively. The exception is for $\epsilon =0.020$ and $\epsilon =0.025$ that this number rises from 41 to 42.
\section*{Acknowledgment}
We would like to thank Azarbaijan Shahid Madani University for its support.

\section{Conclusion}
In this paper, we considered a uni-parametric linear program when an identical parameter linearly perturbed the left and the right sides of constraints. The induced optimal partition was defined, and a  methodology for identifying the corresponding invariancy interval was provided. It was proved that the optimal value function is fractional on each interval; it is continuous at internal jointing points. A computational algorithm was presented, enabling to find all invariancy intervals.  In addition to the traditional concept of the transition point, the concept of change point was also introduced.    Examples indicated the validity of the findings. As future work, this study could be conducted for more than one parameter or a uni-parametric linear program when left and right-hand-side of constraints in addition to the objective coefficients were linearly perturbed by an identical parameter, as well as for the case where the perturbation is not linear.


\begin{thebibliography}{99}

\bibitem{BGK79} H. Bart, I. Gohberg, M. A. Kaashoek, {\em Minimal Fractorization of matrix and operator functions, } Operator Theory: Advances and Aplications 1, Birkh\"{a}ser Verlag, Basel, (1979).

 \bibitem{ben1992volume} A. Ben-Israel, {\em A volume associated with $m \times n$ matrices, }  Linear Algebra and its Applications, \textbf{167} (1992), pp. 87-111.

 \bibitem{ben2003generalized} A. Ben-Israel,  and T.N. Greville,  {\em Generalized inverses: theory and applications,}  Springer Science \& Business Media \textbf{15} (2003).

 \bibitem{BT90} A. Ben-Tal  and M. Teboulle, {\em A geometric property of the least squares solution of linear equations, } Linear Algebra and its Applications, \textbf{139} (1990), pp. 165-170.

\bibitem{BV04}  S. Boyd  and  L. Vandenberghe, {\em Convex optimization, } Econometrica, Cambridge University Press, (2004).


\bibitem{CL17}
V.M. Charitopoulos, L.G. Papageorgiou, and  V. Dua,  {\em  Multi-parametric linear programming under global uncertainty}. AIChE Journal \textbf{63} (9), (2017a. ), pp. 3871-3895 .

\bibitem{FR85}  R. M. Freund, {\em Postoptimal analysis of a linear program under simultaneous changes in matrix coefficients, } Mathematical Programming Study, \textbf{24} (1985), pp. 1-13.

\bibitem{GM15} A. Ghaffari Hadigheh  and N. Mehanfar, {\em Matrix perturbation and optimal partition invariancy in linear optimization, } Asia-Pacific Journal of Operational Research, No. 03, \textbf{32} (2015),  Pages 17.

\bibitem{ART7} A. Ghaffari-Hadigheh, O. Romankko, and T. Terlaky, {\em Sensitivity analysis in convex quadratic optimization: Simultaneous perturbation of the objective and right-hand-side vectors, } Algorithmic Operations Research \textbf{2} (2007), pp. 94-111.

\bibitem{GT56} A. Goldman, A. Tucker, {\em Theory of linear programming, } In H. Kuhn, A. Tucker(eds), Linear Inequalities and Related Systems, Annals of Mathematical Studies, Princeton University Press, Princeton, New Jersey, No. 38, (1956), pp. 53-97.

\bibitem{HG99} H. Greenberg, {\em Matrix sensitivity analysis from an interior solution of a linear program, }  INFORMS Journal on Computing, No. 3, \textbf{11} (1999), pp. 316-327 .

   \bibitem{HG00} H. Greenberg, {\em Simultaneous primal-dual right-hand-side sensitivity analysis from a strictly complementary solution of a linear program, } SIAM Journal of optimization \textbf{10}(2000), pp. 427-442.

   \bibitem{AH10} A. Holder, {\em Parametric linear programming, } (2010).

\bibitem{KK14}
R. Khalilpour, I.A. Karimi, {\em Parametric optimization with uncertainty on the left-hand-side of linear programs,} Comput. Chem. Eng. \textbf{60} (2014), pp. 31-40.

\bibitem{RT97} C. Roos, T. Terlaky, and J. P. Vial,
{\em Interior point methods for linear optimization, } Springer Science \& Business Media, (2005).

\bibitem{zuidwijk2005linear} R. A. Zuidwijk, {\em Linear parametric sensitivity analysis of the constraint coefficient matrix in linear programs, }
ERIM report series research in management, (2005), pp. 1-10.


\end{thebibliography}
\end{document}